\documentclass[journal]{IEEEtran}
\usepackage{xcolor}
\usepackage{amssymb}
\usepackage{amsmath}
\usepackage{amsthm}
\usepackage{lineno}
\usepackage{graphicx}
\usepackage{hyperref}

\usepackage{cite}
\usepackage{algorithm}
\usepackage{algorithmicx}
\usepackage{algpseudocode}
\usepackage{array}
\usepackage[caption=false,font=normalsize,labelfont=sf,textfont=sf]{subfig}
\usepackage{fixltx2e}
\usepackage{url}
\usepackage{booktabs}

\newtheorem{theorem}{Theorem}
\newtheorem{lemma}{Lemma}
\newtheorem{assumption}{Assumption}
\newtheorem{proposition}{Proposition}
\newtheorem{corollary}{Corollary}
\newtheorem{remark}{Remark}
\newcommand\ta{\widetilde{\alpha}}
\newcommand\tu{\widetilde{\mu}}

\newcommand\bz{\mathbf{z}}
\newcommand\bq{\mathbf{q}}
\newcommand\mE{\mathbb{E}}
\newcommand\bg{\mathbf{g}}
\newcommand\bv{\mathbf{v}}
\newcommand\bx{\mathbf{x}}
\newcommand\bW{\mathbf{W}}
\newcommand\bd{\mathbf{d}}
\newcommand\bH{\mathbf{H}}
\newcommand\bu{\mathbf{u}}
\newcommand\btau{\mathbf{\tau}}
\newcommand\ox{\overline{\mathbf{x}}}
\newcommand\og{\overline{\mathbf{g}}}
\newcommand\od{\overline{\mathbf{d}}}
\newcommand\oH{\overline{\mathbf{H}}}

\newcommand\ov{\overline{\mathbf{v}}}
\newcommand\otau{\overline{\mathbf{\tau}}}
\begin{document}
%
\title{Variance-Reduced Stochastic Quasi-Newton Methods for Decentralized Learning: Part I}
\author{Jiaojiao Zhang{$^1$}, Huikang Liu{$^2$}, Anthony Man-Cho So{$^1$}, and Qing Ling{$^3$}
	\thanks{Jiaojiao Zhang and Anthony Man-Cho So are with the Department of Systems Engineering and Engineering Management, The Chinese University of Hong Kong.}%
	\thanks{Huikang Liu is with the Business School, Imperial College London.}
	\thanks{Qing Ling is with the School of Computer Science and Engineering and Guangdong Province Key Laboratory of Computational Science, Sun Yat-Sen University, as well as the Pazhou Lab.}%
}
\maketitle

\begin{abstract}
In this work, we investigate stochastic quasi-Newton methods for minimizing a finite sum of cost functions over a decentralized network. In Part I, we develop a general algorithmic framework that incorporates stochastic quasi-Newton approximations with variance reduction so as to achieve fast convergence. At each time each node constructs a local, inexact quasi-Newton direction that asymptotically approaches the global, exact one. To be specific, (i) A local gradient approximation is constructed by using dynamic average consensus to track the average of variance-reduced local stochastic gradients over the entire network; (ii) A local Hessian inverse approximation is assumed to be positive definite with bounded eigenvalues, and how to construct it to satisfy these assumptions will be given in Part II. Compared to the existing decentralized stochastic first-order methods, the proposed general framework introduces the second-order curvature information without incurring extra sampling or communication. With a fixed step size, we establish the conditions under which the proposed general framework linearly converges to an exact optimal solution.
\end{abstract}


%
%
\IEEEpeerreviewmaketitle

\section{Introduction}
There has been steadily growing interest in machine learning over networks in various areas, such as large-scale learning \cite{lian2017can,assran2019stochastic,amiri2020machine,lee2020optimization}, privacy-preserving learning \cite{brisimi2018federated,warnat2021swarm}, decentralized system control \cite{molzahn2017survey,li2017dynamical}, etc. These applications often can be formulated as decentralized learning problems. In this paper, we focus on a decentralized learning problem over an undirected and connected network, where $n$ nodes cooperatively look for a minimizer of the average cost
\begin{align}\label{c1}
{x}^* =\arg\min_{x\in\mathbb{R}^{d}} ~ F(x) \triangleq  \frac{1}{n} \sum_{i=1}^{n} f_{i}(x).
\end{align}
Here, $x$ is the decision variable and $f_{i}: \mathbb{R}^d\to \mathbb{R}$ is the local cost function of node $i$, represented as the average of $m_i$ sample costs in the form of \begin{align}\label{samplecost}
f_i(x)\triangleq \frac{1}{m_i} \sum_{l=1}^{m_i} f_{i,l}(x),
\end{align}
in which $f_{i,l}: \mathbb{R}^d\to \mathbb{R}$ is the $l$-th sample cost on node $i$, assumed to be differentiable. To agree on an optimal solution ${x}^*$ to \eqref{c1}, the nodes are allowed to communicate with their neighbors and perform local computation. However, at each time each node $i$ is unable to access the local cost function $f_i$ or the local full gradient $\nabla f_i$ since they may involve a large number of samples. Instead, at each time each node $i$ can access one or a mini-batch of sample costs and compute the local stochastic gradient.

In recent years, a large number of algorithms have been proposed for solving \eqref{c1}. Among them, decentralized stochastic first-order methods are appealing due to their low computation complexity. In contrast, decentralized stochastic second-order methods are rarely studied. The goal of this paper is to devise computationally affordable decentralized stochastic second-order methods to accelerate the learning process.



\subsection{Decentralized Deterministic Algorithms}

For the decentralized learning problem \eqref{c1}, when the numbers of local samples $m_i$ are sufficiently small such that each node $i$ is affordable to compute the local full gradient $\nabla f_i$ or even the local full Hessian $\nabla^2 f_i$, there are many decentralized deterministic first-order and second-order algorithms.

Distributed gradient descent (DGD) is a popular first-order method that combines local gradient descent with average consensus, but is unable to achieve exact convergence when using a fixed stepsize \cite{nedic2009distributed,yuan2016convergence}. This convergence error can be eliminated with local historic information, for example, in exact first-order algorithm (EXTRA) \cite{shi2015extra,shi2015proximal}, primal-dual methods \cite{chang2014multi,shi2014linear,zhang2020penalty}, exact diffusion \cite{yuan2018exact1,yuan2018exact}, and gradient tracking \cite{xu2015augmented,qu2017harnessing,sun2019convergence}.

Although the first-order algorithms are widely used in decentralized learning due to their low computation complexity, second-order methods are able to achieve faster convergence. Several works penalize the consensus constraints (namely, all the local decision variables must be eventually consensual) to the cost function, and obtain approximate Newton directions that are computable in decentralized manners \cite{mokhtari2016network,bajovic2017newton,mansoori2019fast}. A decentralized quasi-Newton method is proposed in \cite{eisen2017decentralized}, also using the idea of penalization. However, with the penalty term, these algorithms only converge to a neighborhood of an optimal solution when using fixed stepsizes. This issue has been addressed in second-order primal-dual methods \cite{mokhtari2016dqm,mokhtari2016decentralized,eisen2019primal,9303887}. A decentralized approximate Newton-type algorithm is proposed in \cite{li2020communication}, adopting the gradient tracking technique such that the local gradients can track the global ones. A cubically-regularized Newton method with gradient tracking is explored in \cite{daneshmand2021newton}, running inexact, preconditioned Newton steps on each node. The work of \cite{zhang2020distributed} proposes a decentralized adaptive Newton method, where at each time each node runs a finite-time consensus inner loop.



\subsection{Decentralized Stochastic Algorithms}
When the numbers of local samples $m_i$ are large, decentralized deterministic algorithms become prohibitive due to the time-consuming computation of local full gradients and Hessians. Hence, decentralized stochastic algorithms, where at each time each node only accesses one or a mini-batch of sample costs and computes the local stochastic gradient, are favorable \cite{chen2012diffusion,tang2018d,pu2019sharp}.
To remedy the gradient noise brought by the local stochastic gradients, variance reduction techniques such as stochastic variance reduced gradient (SVRG) can be applied \cite{mokhtari2016dsa,xin2020variance,li2020optimal,pu2020distributed,hendrikx2020optimal,hendrikx2020dual,li2020communication}.
However, to the best of our knowledge, computationally affordable decentralized stochastic second-order methods have not been investigated.

In Part I of this work, we propose a general algorithmic framework that incorporates decentralized stochastic quasi-Newton approximations with variance reduction so as to achieve fast convergence. To be specific, at each time, each node first uses a variance reduction technique (for example, SVRG) to obtain a local corrected stochastic gradient, and then uses the dynamic average consensus method \cite{zhu2010discrete} to obtain an approximation of the global gradient. Further, the gradient approximations are used to construct the Hessian inverse approximations. We prove that if the constructed Hessian inverse approximations have bounded positive eigenvalues, then the proposed general framework converges linearly to an exact optimal solution of \eqref{c1}.

Note that using the gradient approximations to construct the Hessian inverse approximations is quite adventurous, since the gradient approximations are not necessarily reliable due to stochastic gradient noise and disagreement among the nodes. Naively adopting centralized quasi-Newton methods may end up with almost-singular Hessian inverse approximations, or even non-positive semidefinite ones. To address this issue, in Part II of this work, we propose two methods, damped regularized limited-memory DFP (Davidon-Fletcher-Powell) and damped limited-memory BFGS (Broyden-Fletcher-Goldfarb-Shanno), which are able to adaptively construct positive definite Hessian inverse approximations. Further, in Part II, we prove that the generated Hessian inverse approximations have bounded positive eigenvalues, which fits into the general framework in Part I. Convergence rates of the proposed general framework and the existing decentralized stochastic methods are summarized in Table \ref{tab1}.

\begin{table*}[h]
	\centering
	\caption{Stochastic gradient computation complexity of decentralized learning methods to reach an $\epsilon$-optimal solution of \eqref{c1}.}
	\begin{tabular}{c|c|c|c}
		\hline
		Algorithm\footnotemark{}  & Step size & Stochastic gradient
		computation complexity  & Batch size  \\ \hline
		DSA \cite{mokhtari2016dsa} & $\alpha=\mathcal{O}\left(\frac{\lambda_{\min}(\tilde{W})}{L\kappa_F}\right)$\footnotemark{}  & $\mathcal{O}\left(\max \left\{m \kappa_F, \frac{\kappa_F^{4}}{1-\sigma}, \frac{1}{(1-\sigma)^{2}}\right\} \log \frac{1}{\epsilon}\right)$ & $b=1$ \\ \hline
		GT-SVRG \cite{xin2020variance} &$\alpha=\mathcal{O}\left(\frac{\left(1-\sigma^{2}\right)^{2}}{L \kappa_F}\right)$  & $\mathcal{O}\left(\left(m+\frac{\kappa_F^{2} \log \kappa_F}{\left(1-\sigma^{2}\right)^{2}}\right) \log \frac{1}{\epsilon}\right)$ & $b=1$ \\ \hline
		GT-SAGA \cite{xin2020variance}  & $\alpha = \min \left\{\mathcal{O}\left(\frac{1}{\mu m}\right), \mathcal{O}\left( \frac{\left(1-\sigma^{2}\right)^{2}}{L \kappa_F}\right)\right\}$ & $\mathcal{O}\left(\max \left\{m,  \frac{\kappa_F^{2}}{\left(1-\sigma^{2}\right)^{2}}\right\} \log \frac{1}{\epsilon}\right)$ & $b=1$ \\ \hline
		VR-DIGing \cite{li2020optimal}  & $\alpha=\mathcal{O}\left(\frac{1}{\max \left\{L,  \frac{\mu}{(1-\sigma)^2}\right\}}\right)$ & $\mathcal{O}\left((m+\kappa_F)\log\frac{1}{\epsilon}
		\right)$ & $b=\frac{\max \left\{L, m \mu\right\}}{\max \left\{L, \frac{\mu}{(1-\sigma)^2}\right\}}$ \\ \hline
			Acc-VR-DIGing \cite{li2020optimal}  & $\alpha=\mathcal{O}(\frac{1}{L})$ & $\mathcal{O}\left((m+\sqrt{m\kappa_F})\log \frac{1}{\epsilon}\right)$  & $b=\sqrt{\frac{m (1-\sigma)^2 \max \left\{L, m \mu\right\}}{ L}}$ \\ \hline
		DFP and BFGS\footnotemark{}  & $	\alpha =\mathcal{O}\left(\frac{(1-\sigma^2)^2}{LM_2\kappa_F\kappa_H} \right)
	$ & $	\mathcal{O}\left(\left(m+\frac{b\cdot\kappa_F^2\kappa_H^2\log\frac{\kappa_F\kappa_H}{1-\sigma^2}}{(1-\sigma^2)^2}\right)\log\frac{1}{\epsilon}\right)$  & $\frac{m-b}{(m-1)b}\le \frac{1}{160}\min\left\{1, \frac{\zeta(1-\sigma^2)^2}{\gamma^2}\right\}$ \\ \hline
	\end{tabular}
	\label{tab1}
\end{table*}
\footnotetext[1]{For all algorithms, we set the numbers of local samples $m_i=m$ for all nodes $i$. The mini-batch sizes are $b_i=b$ for all nodes $i$. $\sigma$ is the second largest singular value of the mixing matrix $W$ defined in Assumption \ref{asm-W}. $L$ and $\mu$ are the smoothness and strong convexity constants, respectively. $\kappa_F=\frac{L}{\mu}$ is the condition number of the cost function $F$. $\gamma$ and $\zeta$ are defined in Theorem \ref{thom-parameter}.}
\footnotetext[2]{Here, $\tilde{W}=\frac{I+W}{2}$.}
\footnotetext[3]{Here, $\kappa_H=\frac{M_2}{M_1}$ is the condition number of $H$ defined in Theorem \ref{thom-parameter}. DFP needs extra $\mathcal{O}(Md^2)$ computation and $\mathcal{O}(d^2+Md)$ storage per iteration, while BFGS needs extra $\mathcal{O}(Md)$ computation and $\mathcal{O}(Md)$ storage per iteration. This will be shown in Part II.}

Throughout Part I and Part II, the cost functions are assumed to be differentiable but not necessarily twice differentiable. In the proposed framework we only use stochastic gradients to estimate Hessians, while in the analysis we do not need to use the true Hessians.

\textbf{Notations.}  $I_d\in \mathbb{R}^{d\times d}$ denotes the $d\times d$ identity matrix, and $1_n\in \mathbb{R}^{n}$ denotes the $n$-dimensional column vector of all ones. $\|\cdot\|$ denotes the Euclidean norm of a vector. $\otimes$ denotes the Kronecker product. For two vectors $x$ and $y$, $x\le y$ denotes that each entry of $x$ is smaller than that of $y$. For two matrices $A$ and $B$, $A\le B$ denotes that each entry of $A$ is smaller than that of $B$. $A\succeq 0$ and $A\succ 0$ refer that the matrix $A$ is positive semidefinite  and  positive definite, respectively.  $A\succeq B$ and $A\succ B$ mean $A-B \succeq 0$ and $A-B\succ 0$, respectively.  $\lambda_{\max}(\cdot)$ and $\lambda_{\min}(\cdot)$ denote the largest and smallest eigenvalues of a matrix, respectively. The $i$-th largest eigenvalue of a matrix is denoted by $\lambda_i(\cdot)$.  $\rho(\cdot)$ denotes the spectral radius of a matrix and  $\| \cdot \|_2$  denotes its spectral norm. For a positive vector $z=\left[z_{1}, \cdots, z_{d}\right]^{\top} \in \mathbb{R}^d$ and an arbitrary vector $a=\left[a_{1}, \cdots, a_{d}\right]^{\top}\in \mathbb{R}^d$,  $\|a\|_{\infty}^{z}=\max _{i}\left|a_{i}\right| / z_{i}$ denotes the weighted infinity norm of vector $a$ and $\|A\|_{\infty}^{{z}}$ is the
weighted infinity norm of matrix $A$ induced by the vector norm $\|\cdot\|_{\infty}^{{z}}$. Define
$\mathbf{W}=W\otimes I_d\in \mathbb{R}^{nd\times nd}$ and  $\bW_{\infty}=\frac{1_n 1_n^T}{n}\otimes I_d\in\mathbb{R}^{nd\times nd}$. Define the aggregated variable  $\mathbf{x}=[x_1;\cdots;x_n]\in \mathbb{R}^{nd}$ for $x_1, \cdots, x_n \in \mathbb{R}^d$, and similar aggregation rules apply to other variables $ \mathbf{d}, \mathbf{g}, \mathbf{v}$ and $\mathbf{\btau}$. Define the average variable over all the nodes at time $k$ as
$\ox^k=\frac{1}{n}\sum_{i=1}^n x_i^k=\frac{1}{n}(1_n^T\otimes I_d)\bx^k\in\mathbb{R}^d$
and similar average rules apply to other variables $\od^k, \og^k, \ov^k$ and $\otau^k$. Define the aggregated gradient $\nabla f(\bx^k)=[\nabla f_1(x_1^k);\cdots; \nabla f_n(x_n^k)]\in \mathbb{R}^{nd}$. Define the average of all the local gradients at the local variables as  $\overline{\nabla f}(\bx^k)=\frac{1}{n} \sum_{i=1}^{n}\nabla f_i(x_i^k)\in\mathbb{R}^d$. Define $\nabla  F\left(\overline{\mathbf{x}}^{k}\right)=\frac{1}{n} \sum_{i=1}^{n}\nabla f_i(\ox^k) \in\mathbb{R}^d$ as the average of all the local gradients at the common average $\ox^k$. Define a block diagonal matrix $\bH^{k}=\operatorname{diag}\{H_i^{k}\}\in \mathbb{R}^{nd\times nd}$ whose $i$-th block is $H_i^{k} \in \mathbb{R}^{d\times d}$ and an average matrix $\oH^k=\frac{1}{n}\sum_{i=1}^{n} H_i^k\in \mathbb{R}^{d\times d}$. Given a random variable $v$, $\mathbb{E} [v]$ denotes the expectation and $\mathbb{E} [v|\mathcal{F}]$ denotes the expectation  conditioned in $\mathcal{F}$.
\section{Problem Formulation and Algorithm Development}
This section begins with introducing the problem formulation of decentralized learning and basic assumptions. Then we propose a general framework of variance-reduced decentralized stochastic quasi-Newton methods.
\subsection{Problem Formulation}
We consider an undirected and connected graph $\mathcal{G}=(\mathcal{V}, \mathcal{E})$ with node set $\mathcal{V}=\{1,\ldots,n\}$ and edge set $\mathcal{E}\subseteq \mathcal{V}\times \mathcal{V}$. Nodes $i \in \mathcal{V}$ and $j \in \mathcal{V}$ are neighbors and allowed to send information to each other if they are connected with an edge $(i,j)\in \mathcal{E}$. Each node $i$ has a local cost function $f_i$ in the form of \eqref{samplecost}, and makes decisions based on stochastic gradients of $f_i$ and information obtained from its neighbors. Define $\mathcal{N}_i$ as the set of neighbors of node $i$ including itself and let ${x}_{i} \in \mathbb{R}^{d}$ be the local copy of decision variable $x$ kept on node $i$. Since the network is bidirectionally connected, the optimization problem in \eqref{c1} is equivalent to
\begin{align}\label{d2}
\mathbf{x}^*= \underset{\mathbf{x}=[x_1;\cdots;x_n]}{{\arg\min}} & ~ f(\mathbf{x})\triangleq\frac{1}{n} \sum_{i=1}^{n} f_{i}\left({x}_{i}\right), \\\nonumber
\text { s.t. } & ~ {x}_{i}={x}_{j}, ~ \forall j \in \mathcal{N}_{i}, ~ \forall i,
\end{align}
in which we know that $\mathbf{x}^*\triangleq[x^*;\cdots;x^*]\in \mathbb{R}^{nd}$ by stacking $n$ vectors $x^*$ to a long column. By observation, \eqref{c1} and \eqref{d2} are equivalent in the sense that the optimal local variables ${x}_{i}^{*}$ of \eqref{d2} are all equal to the optimal argument $x^*$ of \eqref{c1}, i.e., ${x}_{1}^*=\cdots={x}_{n}^*=x^*$.

Throughout this paper, we make the following assumptions on the cost functions.
\begin{assumption}[Convexity and $L$-smoothness]\label{asm-lips}
	Each local sample cost $f_{i,l}$ is convex and has Lipschitz continuous gradients, i.e.,
\begin{align}
    f_{i,l}(y)\geq & f_{i,l}(x) + \nabla f_{i,l}(x)^T(y-x), \label{ineq:convexity} \\
	f_{i,l}(y)\leq & f_{i,l}(x) + \nabla f_{i,l}(x)^T(y-x) + \frac{L}{2}\|y-x\|^2, \label{ineq:Lsmooth}
\end{align}
for all $x,y\in \mathbb{R}^{d}$ and $L>0$ is the Lipschitz constant.
\end{assumption}

Assumption \ref{asm-lips} implies that the local cost functions $f_i$ defined in \eqref{samplecost} and the global cost function $F$ defined in \eqref{c1} are also convex and $L$-smooth.
Under Assumption \ref{asm-lips}, we have the following lemma, which will be used in the later analysis.
\begin{lemma}\label{lem:Lsmooth}
	Under Assumption \ref{asm-lips}, for all $x, y\in\mathbb{R}^d$, we have
	\begin{align}
	 & \frac{1}{2L}\left\|\nabla f_{i,l}(x)-\nabla f_{i,l}(y)\right\|^2 \\
\leq & f_{i,l}(x)- f_{i,l}(y) - \nabla f_{i,l}(y)^T(x-y). \notag
	\end{align}
\end{lemma}
\begin{proof}
	We consider the function
	\begin{equation}
		\phi(x)=f_{i,l}(x) - f_{i,l}(y) - \nabla f_{i,l}(y)^T(x-y).
	\end{equation}
	It is easy to see that $\phi$ is convex and achieves the global minimal value $0$ at the point $y$ because $\nabla\phi(y)=0$. In addition, $\phi(x)$ is also $L$-smooth, such that
	\begin{equation}
		\begin{split}
		0&\leq\min_{\eta}\left\{\phi(x-\eta\nabla\phi(x))\right\}\\
		&\leq\min_{\eta}\left\{\phi(x)-\eta\|\nabla\phi(x)\|^2+\frac{L\eta^2}{2}\|\nabla\phi(x)\|^2\right\}\\
		&=\phi(x)-\frac{1}{2L}\|\nabla\phi(x)\|^2,
	\end{split}
	\end{equation}
By substituting the definition of $\phi$ into the above inequality, we complete the proof.
\end{proof}

Note that similar arguments also hold for the local cost functions $f_i$ and the global cost function $F$.

\begin{assumption}[$\mu$-strong convexity]\label{asm-strongly}
	The global cost function $F$ is strongly convex, i.e.,
	\begin{align}
	F(y)\geq F(x) + \nabla F(x)^T(y-x) + \frac{\mu}{2}\|y-x\|^2, \label{ineq:strongconvexity}
\end{align}
for all $x,y\in \mathbb{R}^{d}$ and $\mu>0$ is the strong convexity constant.
\end{assumption}

To reach the consensual and optimal solution, the nodes need to mix their local decision variables with those from their neighbors according to predefined weights. Let $w_{i j} \geq 0$ represent the weight that node $i$ assigns to node $j$ and define the mixing matrix $W=[w_{ij}]\in \mathbb{R}^{n\times n}$, which satisfies the following assumption.

\begin{assumption}[Mixing matrix]\label{asm-W}
	The mixing matrix $W$ is  nonnegative with $w_{ij}\ge 0$. The weight $w_{ij} = 0$ if and only if $j \notin $ $\mathcal{N}_{i}$. $W$ is symmetric and doubly stochastic, i.e., $W=W^T$ and $W1_n=1_n$. The null space of $I_n-W$ is $\operatorname{span}\left(1_n\right)$.
\end{assumption}

Mixing matrices satisfying Assumption \ref{asm-W} are common in the literature of decentralized learning over an undirected and connected network; see, e.g., \cite{boyd2004fastest, nedic2018network} for details. According to the Perron-Frobenius theorem \cite{pillai2005perron}, Assumption \ref{asm-W} implies that the eigenvalues of $W$ lie in $(-1,1]$ and the multiplicity of eigenvalue 1 is one. It also implies that the second largest singular value $\sigma$ of $W$ is less than 1, i.e.,
$$\sigma=\| W-\frac{1}{n}1_n 1_n^T\|_2<1.$$

\subsection{A General Framework of Variance-Reduced Decentralized Stochastic Quasi-Newton Methods}
We propose a general framework of decentralized stochastic quasi-Newton methods combined with variance reduction to solve \eqref{d2}.
In the proposed framework, node $i$ updates its local decision variable $x_i^{k+1}$ according to the following decentralized stochastic quasi-Newton step
\begin{equation}\label{qn-1}
x_i^{k+1}=\sum_{j=1}^{n} w_{ij}x_j^k-\alpha d^k_i,
\end{equation}
where $k$ is the time and $\alpha>0$ is a constant stepsize. If the local cost function $f_i$ is twice differentiable, one ideal choice of the direction $d_i^{k+1}$ is the global negative Newton direction $\left( \frac{1}{n} \sum_{i=1}^{n} \nabla^2 f_i(\ox^{k+1})\right)^{-1}\left(  \frac{1}{n} \sum_{i=1}^{n} \nabla f_i(\ox^{k+1})\right) = \left( \nabla^2 F(\ox^{k+1})\right)^{-1} \nabla F(\ox^{k+1})$, i.e., multiplication of the global Hessian inverse and the global gradient at the average variable $\ox^{k+1}=\frac{1}{n}\sum_{i=1}^{n}x_i^{k+1}$. However, computing the global negative Newton direction is expensive in the decentralized stochastic learning setting for two reasons. First, computing the global Hessian inverse and the global gradient is impossible, since each node only has access to the information from itself and its neighbors, instead of from the entire network. Second, even computing the local Hessian inverse and the local gradient is unaffordable  because they involve all the sample costs on each node.

In our proposed general framework, we update the direction $d_i^{k+1}$ with carefully constructed Hessian inverse approximation $H_i^{k+1}$ and gradient approximation $g_i^{k+1}$, given by
\begin{equation}\label{qn-2}
d_i^{k+1}=H_i^{k+1} g_i^{k+1}.
\end{equation}
Compared with the global negative Newton direction, the Hessian inverse approximation $H_i^{k+1}$ is to estimate the global Hessian inverse $\left(\frac{1}{n}\sum_{i=1}^{n}  \nabla^2 f_i(\ox^{k+1})\right)^{-1} = \left( \nabla^2 F(\ox^{k+1})\right)^{-1}$ and the gradient approximation $g_i^{k+1}$ is to estimate the global gradient $\frac{1}{n}\sum_{i=1}^{n} \nabla f_i(\ox^{k+1}) = \nabla F(\ox^{k+1})$. We will construct $g_i^{k+1}$ below and $H_i^{k+1}$ in Part II of this work.

For the gradient approximation $g_i^{k+1}$, we consider the case that the sample size $m_i$ is too large such that it is unaffordable to compute the local full gradient $\nabla f_i(x_i^{k+1})$. Thus, node $i$ uniformly randomly chooses a subset $S_i^{k+1} \subseteq \{1,\ldots,m_i\}$ with cardinality $|S_i^{k+1}|=b_i$, computes their stochastic gradients $\nabla f_{i,l}(x_i^{k+1})$, and obtains a corrected stochastic gradient $v_i^{k+1}$ with SVRG, as
\begin{align}\label{qn-3}
v_i^{k+1}=&\frac{1}{b_i}\sum_{l\in S^{k+1}_i}\Big(\nabla f_{i,l}(x_i^{k+1})-\nabla f_{i,l}(\tau_i^{k+1})\Big) \\
& +\nabla f_{i}(\tau_i^{k+1}), \nonumber
\end{align}
where $\tau_i^{k+1}\in\mathbb{R}^d$ is an auxiliary variable. Given a positive integer $T$, $\tau_i^{k+1}=x_i^{k+1}$ if \hspace{-1em} $\mod(k+1,T)=0$ and $\tau_i^{k+1}=\tau_i^{k}$ otherwise. Therefore, node $i$ calculates its local full gradient once at every $T$ times, and saves it to correct the consequent $T$ local stochastic gradients.  With SVRG,  $v_i^{k+1}$ is a reliable, unbiased estimate to the local full gradient $\nabla f_i(x_i^{k+1})$. However, we expect $g_i^{k+1}$ to estimate the global gradient $\frac{1}{n} \sum_{i=1}^{n}    \nabla f_i(\ox^{k+1})$. Inspired by the gradient tracking strategy \cite{zhu2010discrete}, we construct $g_i^{k+1}$ with a dynamic average consensus step, given by
\begin{equation}\label{qn-4}
g_i^{k+1}=\sum_{j=1}^{n} w_{ij}g_j^{k}+v_i^{k+1}-v_i^k,
\end{equation}
with initialization $g_i^0=v_i^0=\nabla f_i(x_i^0)$.

Intuitively, the local corrected stochastic gradient $v_i^{k+1}$ will gradually approach the local full gradient $\nabla f_i(x_i^{k+1})$ with the help of SVRG. If the local decision variables $x_i^{k+1}$ are almost consensual, then the gradient approximations $g_i^{k+1}$ will gradually approach the global gradient $\frac{1}{n} \sum_{i=1}^{n}  \nabla f_i(\ox^{k+1})$ with the help of dynamic average consensus. With the gradient approximations $g_i^{k+1}$ at the hand of each node $i$, in Part II, we will introduce how to obtain the Hessian inverse approximation $H_i^{k+1}$ that estimates the global Hessian inverse $\left( \frac{1}{n} \sum_{i=1}^{n}  \nabla^2 f_i(\ox^{k+1})\right)^{-1}$. As we will show in Part II, the Hessian inverse approximations $H_i^k$ is constructed locally given $g_i^k$ and $x_i^k$, without extra sample or communication. Just as the first-order gradient tracking methods, the proposed second-order methods require two rounds of communication of  $d$-dimension vectors at each iteration.


\floatname{algorithm}{Algorithm}
\begin{algorithm}[htbp]
	\caption{Variance-reduced decentralized stochastic quasi-Newton methods on node $i$} \label{alg-framework}
	\begin{algorithmic}[1]
		\Require $\alpha$; $T$; $b_i$; $x_i^0$; $d_i^0$; $\tau_i^0=x_i^0$;  $g_i^0=v_i^0=\nabla f_i(x_i^0)$.
		\For {$k =0,1,2,\ldots$}
		\State  Update local decision variable $x_i^{k+1}$ as in \eqref{qn-1}.
		\State Select $S_i^{k+1} \subseteq \{1,\ldots,m_i\}$ with batch size $b_i$.
		\State $\tau_i^{k+1}=\tau_i^{k}$, or $\tau_i^{k+1}=x_i^{k+1}$ if \hspace{-1em} $\mod(k+1,T)=0$.
		\State Update corrected stochastic gradient $v_i^{k+1}$ as in \eqref{qn-3}.
		\State Update gradient approximation $g_i^{k+1}$ as in \eqref{qn-4}.
		\State Construct Hessian inverse approximation $H_i^{k+1}$. 
		\State Update direction $d_i^{k+1}$ as in \eqref{qn-2}.
		\EndFor
	\end{algorithmic}
\end{algorithm}

The proposed variance-reduced decentralized stochastic quasi-Newton methods are described in Algorithm \ref{alg-framework}. The general framework can be written in a compact form of
\begin{align}\label{eq:alg}
\begin{aligned}
\mathbf{x}^{k+1}&=\mathbf{W}\mathbf{x}^{k}-\alpha \mathbf{d}^{k},\\
\mathbf{g}^{k+1}&=\mathbf{W}\mathbf{g}^{k}+\mathbf{v}^{k+1}-\mathbf{v}^{k},\\
\mathbf{d}^{k+1}&=\bH^{k+1}\mathbf{g}^{k+1}.\\
\end{aligned}
\end{align}

\section{Convergence Analysis}
This section establishes the linear convergence rate of the general framework in Algorithm \ref{alg-framework}, given that the Hessian inverse approximations $H_i^{k}$ satisfy the following assumption.
\begin{assumption}[Bounded Hessian inverse approximations]\label{asm-Hbound}
	There exist two constants $M_1$ and $M_2$ with $0< M_1\leq M_2<\infty$ such that
	\begin{equation}\label{b1}
	M_1I_d\preceq H^{k}_i\preceq M_2I_d, \; \forall i=1,\ldots,n,  \; \forall k\ge 0.
	\end{equation}
\end{assumption}

We will not describe how to construct $H_k^i$ until in Part II, where specific updating schemes for $H_i^k$ satisfying Assumption \ref{asm-Hbound} will be proposed. 

\subsection{Preliminaries}
We start the convergence analysis of the general framework with several preliminaries.

First of all, we have the following ``averaging'' property of the mixing step
\begin{align}\label{b11}\nonumber
\|\bW\bx^k-\bW_{\infty}\bx^k\|=&\left\|\left((W-\frac{1}{n}1_n1_n^T)\otimes I_d\right)(\bx^k-\bW_{\infty}\bx^k)\right\|\\
\le& \sigma\|\bx^k-\bW_{\infty}\bx^k\|.
\end{align}
By Assumption \ref{asm-W}, we know $0\le \sigma<1$. Thus, \eqref{b11} implies that $\bW \bx^k$ is closer to the average $\bW_{\infty}\bx^k$ than the unmixed $\bx^k$. This ``averaging'' property will be frequently used over $\bx^k$ and other variables in the analysis.

We recall that in \eqref{eq:alg} the gradient approximation $\bg^k$ is updated by dynamic average consensus \cite{zhu2010discrete}. Under the initialization $\bg^0=\bv^0=\nabla f(\bx^0)$, taking average over all the nodes and using induction \cite{qu2017harnessing}, we have
\begin{equation}\label{b12}
\og^k=\ov^k, ~\forall k.
\end{equation}
It implies that each $g_i^k$ approximately tracks the average of gradient estimators $v_i^k$ when all $g_i^k$ are almost consensual.

To handle the randomness caused by sampling, we
denote $\mathcal{F}^k$ as the history of the dynamical system generated
by $\bigcup_{i=\{1,\cdots,n\}}^{t\le k-1} S_i^t$. For each node $i$, the stochastic vector $v_i^k$ is an  unbiased estimator of local gradient $\nabla f_i(x_i^k)$ conditioned in $\mathcal{F}^k$. Thus we have
\begin{equation}\label{b13}
\mathbb{E}\left[v_i^{k} | \mathcal{F}^{k}\right]=\nabla f_i\left(x_i^{k}\right).
\end{equation}
Further, by \eqref{b12} and \eqref{b13} we have
\begin{equation}\label{b14}
\mathbb{E}\left[\og^{k} | \mathcal{F}^{k}\right]=\mathbb{E}\left[\ov^{k} | \mathcal{F}^{k}\right]=\overline{\nabla f}\left(\mathbf{x}^{k}\right).
\end{equation}
Under Assumption \ref{asm-lips}, we have
\begin{equation}\label{b15}
\left\|\overline{\nabla f}\left(\mathbf{x}^{k}\right)-\nabla F\left(\overline{\mathbf{x}}^{k}\right)\right\| \leq \frac{L}{\sqrt{n}}\left\|\mathbf{x}^{k}-\bW_{\infty} \mathbf{x}^{k}\right\|,~\forall k.
\end{equation}
The proof can be found in Lemma 8 of \cite{qu2017harnessing}.

\subsection{Main Theorem}
Motivated by the analysis in \cite{qu2017harnessing,xin2020variance}, we will use the consensus error $\mE[\|\bx^k-\bW_{\infty} \bx^k\|^2]$, network optimality gap $\mE[F(\ox^k)-F(x^*)]$ and the gradient tracking error $\mE[\|\bg^k-\bW_{\infty}\bg^k\|^2]$ to establish the convergence rate. Collect all the three errors mentioned above into a vector $\bu^k\in \mathbb{R}^3$  such that
\begin{equation*}
\bu^k=\begin{bmatrix}
\mE[ \|\bx^k-\bW_{\infty} \bx^k\|^2 ]\\ \frac{2n}{L}\mE\left[F(\ox^{k})-F(x^*)\right]\\\frac{1-\sigma^2}{L^2}\mE\left[\|\bg^{k}-\bW_{\infty}\bg^k\|^2\right]
\end{bmatrix}.
\end{equation*}
Thus, $\bu^k$ is a distance measure between $x_i^k$ and $x^*$ since $\bu^k=0$ implies $x_i^k=\ox^k=x^*, \forall i$.
Define
\begin{equation}\label{p18}
 B=\max_{i\in\{1,\ldots,n\}}\left\{\frac{m_i-b_i}{(m_i-1)b_i}\right\}<1.
\end{equation}
We call $B$ the non-sampling rate, since $B=0$ means that each node $i$ uses all $m_i$ samples to compute the local full gradient $\nabla f_i$, degenerating to the deterministic setting.

We first summarize the conditions of the parameters, including the step size $\alpha$, the non-sampling rate $B$ and the period $T$ of SVRG, to guarantee the linear convergence rate, in the following theorem.
\begin{theorem}\label{thom-parameter}
	Under Assumptions \ref{asm-lips}--\ref{asm-Hbound}, if the parameters satisfy
	\begin{align}\label{ineq:stepsize}\nonumber
			\alpha&\leq\frac{(1-\sigma^2)^2\mu M_1}{200L^2M_2^2},\\
				B&\le \frac{1}{160}\min\left\{1, \frac{\zeta(1-\sigma^2)^2}{\gamma^2}\right\},\\\nonumber
				T&\ge\frac{2\log(280/(\zeta(1-\sigma^2)^2))}{\zeta \ta },
	\end{align}
where
\begin{equation}
	\zeta = \left(\frac{\mu}{L}\right)^2\left(\frac{M_1}{M_2}\right)^2,\quad
	\gamma=1-\frac{M_1}{M_2},\quad
	\ta=\frac{M_2^2L^2}{M_1\mu}\alpha.
\end{equation}
	Then, the proposed Algorithm \ref{alg-framework} converges linearly to the optimal solution of \eqref{d2}.
\end{theorem}

Theorem \ref{thom-parameter} implies that if the step size $\alpha$ and the non-sampling rate $B$ are small enough and the period $T$ of SVRG is sufficiently large, then the proposed Algorithm \ref{alg-framework} converges at a linear rate to the optimum. Substituting $$\alpha=\mathcal{O}\left(\frac{(1-\sigma^2)^2\mu M_1}{L^2M_2^2}\right),$$ we can see that $$T=\mathcal{O}\left(\frac{\kappa_F^2\kappa_H^2\log\frac{\kappa_F\kappa_H}{1-\sigma^2}}{(1-\sigma^2)^2}\right),$$ where $\kappa_F=L/\mu$ denotes the condition number of  global cost function $F$, $\kappa_H=M_2/M_1$ denote the condition number of Hessian inverse approximations, and $1-\sigma^2$ represents the connectedness of the network. Therefore, to achieve an $\epsilon$-optimal solution, the total number of stochastic gradient evaluations required by Algorithm \ref{alg-framework} is
\begin{equation*}
	\mathcal{O}\left(\left(\max_i\{m_i\}+\frac{\max_i\{b_i\}\cdot\kappa_F^2\kappa_H^2\log\frac{\kappa_F\kappa_H}{1-\sigma^2}}{(1-\sigma^2)^2}\right)\log\frac{1}{\epsilon}\right).
\end{equation*}

\begin{remark}
When $m_i=m$ and $H_i^k=I_d$, we have $\gamma=0$, $b_i=\mathcal{O}(1)$ and $\kappa_H=1$, such that the number of stochastic gradient evaluations of Algorithm \ref{alg-framework} is $$	\mathcal{O}\left(\left(m+\frac{\kappa_F^2\log\frac{\kappa_F}{1-\sigma^2}}{(1-\sigma^2)^2}\right)\log\frac{1}{\epsilon}\right),$$ which is similar to $$\mathcal{O}\left(\left(m+\frac{\kappa_F^{2} \log \kappa_F}{\left(1-\sigma^{2}\right)^{2}}\right) \log \frac{1}{\epsilon}\right),$$ given in \cite{xin2020variance}. Our analysis cannot show better dependence on $\kappa_F$, because we only  assume that $H_i^k$ have bounded positive eigenvalues in the general framework, which includes the worst case, for example, when $H_i^k$ fails to involve any curvature information. The work of \cite{berglund2021distributed} studies deterministic, quadratic cost functions and shows that communicating second-order information helps obtain better dependence on the condition number of cost functions. Our work investigates stochastic, general cost functions and does not communicate Hessians.
Besides, it is recommended to set the batch sizes $b_i=1$ for the variance-reduced stochastic first-order method in \cite{xin2020variance}, since smaller batch sizes reduce the number of stochastic gradient evaluations per iteration, although increasing the number of iterations. For the proposed variance-reduced stochastic second-order methods, the above analysis also suggests to use moderate batch sizes. However, numerical experiments in Part II will show that slightly larger batch sizes are beneficial. Our conjecture is that using larger batch sizes enables constructing stabler gradient and Hessian estimators, and hence helps convergence.
\end{remark}

\subsection{Linear Convergence Rate}
Next, we will specify the linear rate given the parameters in Theorem \ref{thom-parameter}. The analysis includes four steps, where the consensus error $\mE[\|\bx^k-\bW_{\infty} \bx^k\|^2]$, the network optimality gap $\mE[F(\ox^k)-F(x^*)]$ and the gradient tracking error $\mE[\|\bg^k-\bW_{\infty}\bg^k\|^2]$ are bounded in Steps I, II and III, respectively. In Step IV, we reorganize the previous three bounds in a compact form and choose the parameters to establish the linear rate.

\subsubsection{\bf Step I}
The following lemma gives a recursion of the consensus error $\mE[\|\bx^k-\bW_{\infty} \bx^k\|^2]$.
\begin{lemma}\label{lem-x}Under Assumptions \ref{asm-lips}--\ref{asm-Hbound}, consider the sequence $\{\bx^k\}$ generated by \eqref{eq:alg}. For all $k\geq 1$, we have
\begin{align}\label{a11}
&\mE\left[\|\bx^{k+1}-\bW_{\infty}\bx^{k+1}\|^2\right]\\\nonumber
\leq& \left(\frac{1+\sigma^2}{2}+\frac{2\alpha^2\gamma^2 M_2^2L^2}{1-\sigma^2}\right)\mE\left[\|\bx^{k}-\bW_{\infty}\bx^{k}\|^2\right]
\\\nonumber
&+\frac{2\alpha^2M_2^2}{1-\sigma^2}\Big(2\mE\left[\|\bg^k-\bW_{\infty}\bg^k\|^2\right]+\frac{\gamma^2}{n}\mE\left[\|\bv^k-\nabla f(\bx^k)\|^2\right]\\\nonumber
&+2\gamma^2Ln\mE\left[F(\ox^k)-F(x^*)\right]\Big).
\end{align}
\end{lemma}

\begin{proof}
According to the update of $\bx^{k+1}$ in \eqref{eq:alg}, we have
\begin{align}\label{a12}
 &\|\bx^{k+1}-\bW_{\infty}\bx^{k+1}\|^2\\\nonumber
=&\|\bW \bx^k-\alpha \bd^k-\bW_{\infty}\bx^{k}+\alpha \bW_{\infty}\bd^k\|^2\\\nonumber
=&\|\bW (\bx^k-\bW_{\infty}\bx^k)-\alpha (I_{nd}-\bW_{\infty})\bd^{k}\|^2\\\nonumber
\le&(1+\eta)\|\bW (\bx^k-\bW_{\infty}\bx^k)\|^2+(1+\frac{1}{\eta})\alpha^2\|(I_{nd}-\bW_{\infty})\bd^{k}\|^2\\\nonumber
\le&\frac{1+\sigma^2}{2\sigma^2}\|\bW (\bx^k-\bW_{\infty}\bx^k)\|^2+\frac{(1+\sigma^2)\alpha^2}{1-\sigma^2}\|(I_{nd}-\bW_{\infty})\bd^{k}\|^2\\\nonumber
\le&\frac{1+\sigma^2}{2}\|\bx^k-\bW_{\infty}\bx^k\|^2+\frac{2\alpha^2}{1-\sigma^2}
\|(I_{nd}-\bW_{\infty})\bd^{k}\|^2,
\end{align}
where we use $\bW_{\infty}=\bW \bW_{\infty}$ in the first and the second equalities, use Young's inequality with parameter $\eta>0$ in the first inequality, and set $\eta=\frac{1-\sigma^2}{2\sigma^2}$ in the second inequality. For the last inequality, we use  $\|\bW(\bx^k-\bW_{\infty}\bx^k)\|^2=\|\bW\bx^k-\bW_{\infty}\bx^k\|^2\le \sigma^2\|\bx^k-\bW_{\infty}\bx^k\|^2$ obtained by \eqref{b11} and the fact of $\sigma< 1$ by Assumption \ref{asm-W}.

Next, we bound the term $\|(I_{nd}-\bW_{\infty})\bd^{k}\|$ in \eqref{a12}. By the triangle inequality, we have
\begin{align}
&\|(I_{nd}-\bW_{\infty})\bd^{k}\|\\
=&\|(I_{nd}-\bW_{\infty})(\bH^k-\bar{M}I_{nd})\bg^{k}+\bar{M}(I_{nd}-\bW_{\infty})\bg^{k}\|\nonumber\\
\le&\frac{M_2-M_1}{2}\|\bg^k\|+\bar{M}\|\bg^k-\bW_{\infty}\bg^k\|,\nonumber
\end{align}
where $\bar{M}=\frac{M_1+M_2}{2}$. The inequality holds because $\|I_{nd}-\bW_{\infty}\|_2\leq 1$ and $M_1I_{nd}\preceq\bH^k\preceq M_2I_{nd}$. Further, we have
\begin{align}
\|\bg^k\|\le&\|\bg^k-\bW_{\infty}\bg^k\|+\sqrt{n}\|\og^k\|\\\nonumber
\le&\|\bg^k-\bW_{\infty}\bg^k\|+\sqrt{n}(\|\og^k-\nabla F(\ox^k)\|+\|\nabla F(\ox^k)\|),
\end{align}
which implies
\begin{align*}
	&\|(I_{nd}-\bW_{\infty})\bd^{k}\| \le M_2\|\bg^k-\bW_{\infty}\bg^k\| \\
	&\hspace{4em}+\frac{M_2\gamma}{2}\sqrt{n}(\|\og^k-\nabla F(\ox^k)\|+\|\nabla F(\ox^k)\|), \nonumber
\end{align*}
where $\gamma=1-M_1/M_2$. By taking square on both sides and applying Cauchy-Schwarz inequality twice on the above inequality, we get
\begin{align}\label{a13}
	&\|(I_{nd}-\bW_{\infty})\bd^{k}\|^2\\\nonumber
	\le&2M_2^2\|\bg^k-\bW_{\infty}\bg^k\|^2\\\nonumber
	&+M_2^2\gamma^2n(\|\og^k-\nabla F(\ox^k)\|^2+\|\nabla F(\ox^k)\|^2).\nonumber
\end{align}

For the sake of handling the second term in \eqref{a13}, we expand $\mE\left[\|\og^k-\nabla F(\ox^k)\|^2|\mathcal{F}^k\right]$ as
\begin{align}\label{a25}
	&\mE\left[\|\og^k-\nabla F(\ox^k)\|^2|\mathcal{F}^k\right]\\\nonumber
	=&\mE\left[\|\ov^k-\overline{\nabla f}(\bx^k)+\overline{\nabla f}(\bx^k)-\nabla F(\ox^k)\|^2|\mathcal{F}^k\right]\\\nonumber
	=&\mE\left[\|\ov^k-\overline{\nabla f}(\bx^k)\|^2+\|\overline{\nabla f}(\bx^k)-\nabla F(\ox^k)\|^2|\mathcal{F}^k\right]\\\nonumber
	\le& \mE\left[\frac{1}{n^2}\|\bv^k-\nabla f(\bx^k)\|^2+\frac{L^2}{n}\|\bx^k-\bW_{\infty}\bx^k\|^2 |\mathcal{F}^k\right],
\end{align}
where we use $\og^k=\ov^k$ obtained by \eqref{b12} in the first equality and
$\mathbb{E} \left[ \left\langle\overline{\nabla f}\left(\bx^{k}\right)-\nabla F\left(\overline{\mathbf{x}}^{k}\right), \ov^k-\overline{\nabla f}(\bx^k) \right\rangle | \mathcal{F}^{k}\right]=0$
by \eqref{b14} in the second equality. For the inequality, in addition to \eqref{b15}, we use the fact that $v_i^k-\nabla f_i(x_i^k)$ is independent for each node $i$ given the history $\mathcal{F}^k$ and thus
\begin{align}\label{f28}\nonumber
&\mE\left[\|\ov^k-\overline{\nabla f}(\bx^k)\|^2|\mathcal{F}^k\right]=\mE\left[\left\|\frac{1}{n}\sum_{i=1}^{n}\left(v_i^k-\nabla f_i(x_i^k)\right)\right\|^2|\mathcal{F}^k\right] \\
&\hspace{4em}=\frac{1}{n^2}\mE\left[\|\bv^k-\nabla f(\bx^k)\|^2|\mathcal{F}^k\right].
\end{align}

For the third term in \eqref{a13}, we apply Lemma \ref{lem:Lsmooth} to get
\begin{equation}\label{ineq:nabla-xk-bar}
	\|\nabla F(\ox^k)\|^2\leq 2L(F(\ox^k)-F(x^*)),
\end{equation}
where we use $\nabla F(x^*)=0$.
Taking conditioned expectation on both sides of \eqref{a13} and substituting \eqref{a25} and \eqref{ineq:nabla-xk-bar} into it, we get
\begin{align}\label{ineq:gk-norm}
		&\mE\left[\|(I_{nd}-\bW_{\infty})\bd^{k}\|^2|\mathcal{F}^k\right]\\\nonumber
		\le&M_2^2\cdot\mE\Big[2 \|\bg^k-\bW_{\infty}\bg^k\|^2+\gamma^2L^2\|\bx^k-\bW_{\infty}\bx^k\|^2\\\nonumber
		+&\frac{\gamma^2}{n} \|\bv^k-\nabla f(\bx^k)\|^2+2\gamma^2Ln \left(F(\ox^k)-F(x^*)\right) |\mathcal{F}^k\Big].
\end{align}
Finally, taking total expectation on  \eqref{ineq:gk-norm} and \eqref{a12}, and then combining the results, we get \eqref{a11} and complete the proof.
\end{proof}

\subsubsection{\bf Step II}
In order to bound the network
	optimality gap $\mE[F(\ox^k)-F(x^*)]$, we first give the following lemma.
\begin{lemma}\label{lem-d}
	Under Assumptions \ref{asm-lips}--\ref{asm-Hbound}, consider the sequence $\{\bd^k\}$ generated by \eqref{eq:alg}. For all $k\geq 1$, we have
	\begin{align}\label{a20}
	&\mE\left[\|\od^{k}-\oH^k\nabla F(\ox^k)\|^2\right]\\\nonumber
	\leq& \frac{2M_2^2}{n}\Big(L^2\mE\left[\|\bx^{k}-\bW_{\infty}\bx^{k}\|^2\right]+\frac{\gamma^2}{4} \mE\left[\|\bg^{k}-\bW_{\infty}\bg^{k}\|^2\right]\\\nonumber
	&+\frac{1}{n}\mE\left[\|\bv^k-\nabla f (\bx^k)\|^2\right]\Big).
	\end{align}
\end{lemma}

\begin{proof}
According to Assumption \ref{asm-Hbound}, we have
\begin{equation}\label{a21}
	\begin{aligned}
	&\mE \left[\|\od^{k}-\oH^k\nabla F(\ox^k)\|^2\right]\\
	=&\mE\left[\|\od^{k}-\oH^k\og^k+ \oH^k \left(\og^k-\nabla F(\ox^k)\right)\|^2\right]\\
	\le& \mE\left[2\|\od^{k}-\oH^k\og^k\|^2+2M_2^2\|\og^k-\nabla F(\ox^k)\|^2\right].
	\end{aligned}
\end{equation}

For the first term in \eqref{a21}, we compute
\begin{align}\label{eq:dkbar}	
\od^k=\frac{1}{n}\sum_{i=1}^{n}H_i^kg_i^k
=&\frac{1}{n}\sum_{i=1}^{n}H_i^k(g_i^k-\og^k)+\oH^k\og^k\\\nonumber
=&\frac{1}{n}\sum_{i=1}^{n}\left(H_i^k-\bar{M}I_{d}\right)(g_i^k-\og^k)+\oH^k\og^k,
\end{align}
where $\bar{M}=\frac{M_1+M_2}{2}$. Thus we have
\begin{align}
\|\od^{k}-\oH^k\og^k\|^2
\le& \frac{1}{n}\sum_{i=1}^{n}\left\|\left(H_i^k-\bar{M}I_{d}\right)(g_i^k-\og^k)\right\|^2 \label{a23} \\
\le&\frac{M^2_2\gamma^2}{4n}\|\bg^k-\bW_{\infty}\bg^k\|^2. \notag
\end{align}
Taking total expectation on both sides of \eqref{a23} and \eqref{a25} and substituting the results into the two terms at the right-hand side of \eqref{a21}, we get \eqref{a20} and complete the proof.
\end{proof}

With Lemma \ref{lem-d}, we are ready to give the recursion of the network optimality gap $\mE[F(\ox^k)-F(x^*)]$ as follows.
\begin{lemma}\label{lem-F}
	Under Assumptions \ref{asm-lips}-- \ref{asm-Hbound}, when the parameters satisfy the conditions in Theorem \ref{thom-parameter}, consider the sequence $\{\bx^k\}$ generated by \eqref{eq:alg}. For all $k\geq 1$, we have
	\begin{align}\label{a26}
	& n\mE\left[F(\ox^{k+1})-F(x^*)\right]\\\nonumber
	\hspace{-1em}\leq& (1\!-\!\tilde{\mu}\alpha)\!\cdot \! n\mE\left[F(\ox^{k})\!-\!F(x^*)\right]
	\!+\!\frac{\alpha M_2^2\eta}{2nM_1}\mE\left[\|\bv^k\!-\!\nabla f(\bx^k)\|^2\right]\\\nonumber
	+&\!\frac{1.01\alpha M_2^2}{M_1}\Big(\!2L^2\mE\left[\|\bx^{k}-\bW_{\infty}\bx^{k}\|^2\right]\!+\!\frac{\gamma^2}{4}\mE\left[\|\bg^{k}-\bW_{\infty}\bg^{k}\|^2\right]\!\Big),
	\end{align}
where we define $\tilde{\mu}=0.99\mu M_1$ and $\eta=\gamma^2+4\alpha LM_1\leq 1.02$.
\end{lemma}

\begin{proof}
Taking the average of the update of $\bx^{k+1}$ in \eqref{eq:alg} over all the nodes, we have
\begin{equation}
	\ox^{k+1}=\ox^k-\alpha\od^k.
\end{equation}
Then we compute the global cost at the average $\ox^{k+1}$ and get
\begin{align}\label{a28}
F(\ox^{k+1})
\le& F(\ox^k)-\alpha\Big\langle \nabla F(\ox^k),\od^k\Big\rangle+\frac{L\alpha^2}{2}\|\od^k\|^2\\\nonumber
\le& F(\ox^k)-\alpha\Big\langle \nabla F(\ox^k),\oH^k\nabla F(\ox^k)\Big\rangle\\\nonumber
&+\alpha\Big\langle \nabla F(\ox^k),\oH^k\nabla F(\ox^k)-\od^k\Big\rangle\\\nonumber
&+\alpha^2L\left(\|\oH^k\nabla F(\ox^k)\|^2+\|\oH^k\nabla F(\ox^k)-\od^k\|^2\right),
\end{align}
where we use Assumption \ref{asm-lips} in the first inequality and Cauchy-Schwarz inequality in the second inequality.

For the second term in the last inequality of \eqref{a28}, we have
\begin{equation}\label{ineq:value_decrease}
	\Big\langle \nabla F(\ox^k),\oH^k\nabla F(\ox^k)\Big\rangle \geq M_1\|\nabla F(\ox^k)\|^2.
\end{equation}
In order to bound the third term in the last inequality of \eqref{a28}, we derive
\begin{align}
\hspace{-1em}& \mE\left[\oH^k\nabla F(\ox^k)-\od^k\mid \mathcal{F}^k \right]\\
\hspace{-1em}=&\mE\left[\oH^k(\nabla F(\ox^k)-\og^k)-\frac{1}{n}\sum_{i=1}^{n}\left(H_i^k-\bar{M}I_{d}\right)(g_i^k-\og^k)\mid\mathcal{F}^k\right] \notag\\
\hspace{-1em}=&\mE\Big[\oH^k\left(\nabla F(\ox^k)-\overline{\nabla f}(\bx^k)\right)+\left(\oH^k-\bar{M}I_d\right)(\overline{\nabla f}(\bx^k)-\ov^k) \notag\\
\hspace{-1em}&-\frac{1}{n}\sum_{i=1}^{n}\left(H_i^k-\bar{M}I_{d}\right)(g_i^k-\og^k)\mid\mathcal{F}^k\Big], \notag
\end{align}
where the first equality holds because of \eqref{eq:dkbar} and the last equality holds because of the fact that $\mE[\og^k|\mathcal{F}^k]=\mE[\ov^k|\mathcal{F}^k]=\overline{\nabla f}(\bx^k)$. Then, using the triangle inequality and the fact that $\|H^k_i-\bar{M}I_d\|_2\leq\frac{{M}_2\gamma}{2}, \forall i$ and  $\|\oH^k-\bar{M}I_d\|_2\leq\frac{M_2\gamma}{2}$, we have
\begin{align}
	&\left\|\mE\left[\oH^k\nabla F(\ox^k)-\od^k\mid \mathcal{F}^k \right]\right\|^2\\\nonumber
	\leq&\mE \Big[4M_2^2\left\|\nabla F(\ox^k)-\overline{\nabla f}(\bx^k)\right\|^2+\frac{M_2^2\gamma^2}{2n}\|\bg^k-\bW_{\infty}\bg^k\|^2\\\nonumber
	&+M_2^2\gamma^2\|\overline{\nabla f}(\bx^k)-\ov^k\|^2 |\mathcal{F}^k\Big]\\\nonumber
	\leq& \frac{M_2^2}{n}\cdot\mE  \Big[4L^2\|\bx^{k}-\bW_{\infty}\bx^{k}\|^2+ \frac{\gamma^2}{2}\|\bg^{k}-\bW_{\infty}\bg^{k}\|^2\\\nonumber
	&+\frac{\gamma^2}{n}\|\bv^k-\nabla f(\bx^k)\|^2|\mathcal{F}^k\Big],
\end{align}
where the last inequality holds because of \eqref{b15} and \eqref{f28}. Thus, applying Cauchy-Schwarz inequality to the third term in \eqref{a28}, we get
\begin{align}\label{ineq:value_gradient_error}
		&\mE\left[\Big\langle \nabla F(\ox^k),\left(\oH^k\nabla F(\ox^k)-\od^k\right)\Big\rangle\mid \mathcal{F}^k\right] \\\nonumber
		\leq&\frac{M_1}{2}\|\nabla F(\ox^k)\|^2+\frac{1}{2M_1}\left\|\mE\left[\oH^k\nabla F(\ox^k)-\od^k\mid \mathcal{F}^k \right]\right\|^2 \notag\\\nonumber
		\leq & \frac{M_1}{2}\|\nabla F(\ox^k)\|^2+\frac{M_2^2}{2nM_1}\cdot\mE  \Big[4L^2\|\bx^{k}-\bW_{\infty}\bx^{k}\|^2\\\nonumber
		&+\frac{\gamma^2}{2}\|\bg^{k}-\bW_{\infty}\bg^{k}\|^2+\frac{\gamma^2}{n}\|\bv^k-\nabla f(\bx^k)\|^2|\mathcal{F}^k\Big]. \notag
\end{align}

By taking total expectation on both sides of \eqref{a28}, as well as substituting \eqref{ineq:value_decrease}, \eqref{ineq:value_gradient_error} and \eqref{a20}, we have
\begin{align}\label{a29}
&\mE\left[F(\ox^{k+1})\right]\\\nonumber
\le& \mE \left[F(\ox^k)\right]-\left(\frac{\alpha M_1}{2}-\alpha^2 LM_2^2\right) \mE \left[\|\nabla F(\ox^k)\|^2\right]\\\nonumber
&+\left(\frac{2\alpha M_2^2}{nM_1}+\frac{2M_2^2\alpha^2L}{n}\right)L^2\mE\left[\|\bx^{k}-\bW_{\infty}\bx^{k}\|^2\right]\\\nonumber
&+\left(\frac{\alpha M_2^2\gamma^2}{2nM_1}+\frac{2M_2^2\alpha^2L}{n}\right)\frac{1}{n}\mE\left[\|\bv^k-\nabla f(\bx^k)\|^2\right]\\\nonumber
&+\left(\frac{\alpha M_2^2\gamma^2}{4nM_1}+\frac{M_2^2\alpha^2\gamma^2L}{2n}\right) \mE\left[\|\bg^{k}-\bW_{\infty}\bg^{k}\|^2\right].
\end{align}

Next, we bound the four coefficients at the right-hand side of \eqref{a29}. The parameters in Theorem \ref{thom-parameter} imply that
\begin{align}\label{a30}\nonumber
		\frac{\alpha M_1}{2}-\alpha^2 LM_2^2&=\frac{\alpha M_1}{2}\left(1-\frac{2\alpha L M_2^2}{M_1}\right)\geq0.495\alpha M_1,\\\nonumber
\frac{2\alpha M_2^2}{nM_1}+\frac{2M_2^2\alpha^2L}{n}&=\frac{\alpha M_2^2}{nM_1}(2+2\alpha M_1L)\le \frac{2.01\alpha M_2^2}{nM_1},\\\nonumber
\frac{\alpha M_2^2\gamma^2}{2nM_1}+\frac{2M_2^2\alpha^2L}{n}&\leq \frac{\alpha M_2^2\eta}{2nM_1}\\
\frac{\alpha M_2^2\gamma^2}{4nM_1}+\frac{M_2^2\alpha^2\gamma^2L}{2n}&\leq \frac{1.01\alpha M_2^2\gamma^2}{4nM_1},
\end{align}
where we use $\alpha M_2 L \le \frac{M_1}{200M_2}$, $\alpha M_1 L\le \frac{1}{200}$ and
$\eta=\gamma^2+4\alpha LM_1\leq 1.02$.
By Assumption \ref{asm-strongly}, we have
\begin{equation}\label{a31}
\|\nabla F(\ox^k)\|^2\ge 2\mu\left[F(\ox^k)-F(x^*)\right].
\end{equation}
Substituting \eqref{a30} and \eqref{a31} into \eqref{a29} and subtracting $F(x^*)$ on both sides, we have
\begin{align}
&\mE\left[F(\ox^{k+1})-F(x^*)\right]\\\nonumber
\leq& \left(1-0.99\alpha M_1\mu\right) \mE\left[F(\ox^{k})-F(x^*)\right]\\\nonumber
&+\frac{\alpha M_2^2\eta}{2n^2M_1}\mE\left[\|\bv^k-\nabla f(\bx^k)\|^2\right]+ \frac{1.01\alpha M_2^2}{M_1n}\times \\\nonumber
&\quad\Big(2L^2\mE\left[\|\bx^{k}-\bW_{\infty}\bx^{k}\|^2\right]+\frac{\gamma^2}{4}\mE\left[\|\bg^{k}-\bW_{\infty}\bg^{k}\|^2\right]\Big).
\end{align}
By multiplying $n$ on both sides and using the definition of $\tilde{\mu}=0.99M_1\mu$, we get \eqref{a26} and complete the proof.
\end{proof}

\subsubsection{\bf Step III}
The following lemma establishes a recursion of the gradient tracking error $\mE[\|\bg^k-\bW_{\infty}\bg^k\|^2]$.
\begin{lemma}\label{lem-g}
	Under Assumptions \ref{asm-lips}-- \ref{asm-Hbound}, when the parameters satisfy the conditions in Theorem \ref{thom-parameter}, consider the sequence $\{\bg^k\}$ generated by \eqref{eq:alg}. For all $k\geq 1$, we have
	\begin{align}\label{a34}
	&\mE\left[\|\bg^{k+1}-\bW_{\infty}\bg^{k+1}\|^2\right]\\\nonumber
	\leq& \frac{1+\sigma^2}{2}\mE\left[\|\bg^{k}-\bW_{\infty}\bg^{k}\|^2\right]+\frac{4L^2}{1-\sigma^2}\mE\left[\|\bx^{k+1}-\bx^k\|^2\right]\\\nonumber
	&+\!\frac{4}{1-\sigma^2}\left(\mE\left[\|\bv^{k+1}\!-\!\nabla f(\bx^{k+1})\|^2\right]\!+\!\mE\left[\|\bv^{k}\!-\!\nabla f(\bx^{k})\|^2\right]\right).
	\end{align}
\end{lemma}
\begin{proof}
By the update of $\bg^{k+1}$ in \eqref{eq:alg}, we have
\begin{align}\label{a35}
&\bg^{k+1}-\bW_{\infty}\bg^{k+1}\\\nonumber
=& \mathbf{W}\mathbf{g}^{k}+\mathbf{v}^{k+1}-\mathbf{v}^{k} - \bW_{\infty}\left(\mathbf{W}\mathbf{g}^{k}+\mathbf{v}^{k+1}-\mathbf{v}^{k}\right)\\\nonumber
=&\bW(\bg^k-\bW_{\infty}\bg^k)+(I_{nd}-\bW_{\infty})(\bv^{k+1}-\bv^k),
\end{align}
where we use $\bW_{\infty}=\bW_{\infty}\bW=\bW\bW_{\infty}$. Taking square on both sides of \eqref{a35}, using Young's inequality with parameter $\eta=\frac{1-\sigma^2}{2\sigma^2}$ and then taking total expectation, we get
\begin{align}\label{a36}
&\mE\left[\|\bg^{k+1}-\bW_{\infty}\bg^{k+1}\|^2\right]\\
\le&\mE\Big[ (1+\eta)\|\bW(\bg^k-\bW_{\infty}\bg^k)\|^2 \notag\\
&\quad+(1+\eta^{-1})\|(I_{nd}-\bW_{\infty})(\bv^{k+1}-\bv^k)\|^2\Big] \notag\\
\le & \mE\Big[\frac{1+\sigma^2}{2}\|\bg^k-\bW_{\infty}\bg^k\|^2+\frac{2}{1-\sigma^2}\|\bv^{k+1}-\bv^k\|^2\Big], \notag
\end{align}
where we use the facts $\|\bW(\bg^k-\bW_{\infty}\bg^k)\|^2\le \sigma^2\|\bg^k-\og^k\|^2$ and $\|I_{nd}-\bW_{\infty}\|_2=1$ in the second inequality.

For the second term at the right-hand side of \eqref{a36}, expand $\mE\left[\|\bv^{k+1}-\bv^k\|^2\right]$  as
\begin{align}\label{a37}
&\mE\left[\|\bv^{k+1}-\bv^k\|^2\right]\\\nonumber
\le&2\mE\left[\|\bv^{k+1}-\bv^k-\left(\nabla f(\bx^{k+1})-\nabla f(\bx^{k})\right)\|^2\right]\\\nonumber
&+2\mE\left[\|\nabla f(\bx^{k+1})-\nabla f(\bx^k)\|^2\right]\\\nonumber
\le&2\mE\left[\|\bv^{k+1}-\nabla f(\bx^{k+1})\|^2\right]+2\mE\left[\|\bv^k-\nabla f(\bx^k)\|^2\right]\\\nonumber
&+2L^2\mE\left[\|\bx^{k+1}-\bx^k\|^2\right],
\end{align}
where in the last inequality, we use Assumption \ref{asm-lips} and the fact that
\begin{align*}
 & \mE\left[\left\langle \bv^{k+1}-\nabla f(\bx^{k+1}),\bv^k-\nabla f(\bx^k) \right\rangle\right] \\
=& \mE\left[\mE\left[\left\langle \bv^{k+1}-\nabla f(\bx^{k+1}),\bv^k-\nabla f(\bx^k) \right\rangle|\mathcal{F}^{k+1}\right]\right]=0.
\end{align*}

By substituting \eqref{a37} into \eqref{a36}, we obtain \eqref{a34} and complete the proof.
\end{proof}

Up to now, we have already bounded the consensus error $\mE[\|\bx^k-\bW_{\infty} \bx^k\|^2]$ in Lemma \ref{lem-x},
the network optimality gap $\mE[F(\ox^k)-F(x^*)]$ in Lemma \ref{lem-F} and the gradient tracking error $\mE[\|\bg^k-\bW_{\infty}\bg^k\|^2]$ in Lemma \ref{lem-g}, respectively.  Observe that the upper bounds in Lemmas \ref{lem-F} and \ref{lem-g} contain the variable difference $\|\bx^{k+1}-\bx^k\|$, as well as the variances of gradient estimators $\mE\left[\|\bv^{k}-\nabla f(\bx^k)\|^2\right]$ and $\mE\left[\|\bv^{k+1}-\nabla f(\bx^{k+1})\|^2\right]$. Below we proceed to further bounding these terms.

The following lemma bounds the difference of two successive iterations  $\|\bx^{k+1}-\bx^k\|$.
\begin{lemma}\label{lem-xk}
Under Assumptions \ref{asm-lips}--\ref{asm-Hbound}, 
when the parameters satisfy the conditions in Theorem \ref{thom-parameter},
consider the sequence $\{\bx^k\}$ generated by \eqref{eq:alg}. For all $k\geq 1$, we have
\begin{align}\label{a38}
&\mE\left[\|\bx^{k+1}-\bx^k\|^2\right]\\\nonumber
\le& 8.01\mE\left[\|\bx^{k}-\bW_{\infty}\bx^k\|^2\right]+\frac{4\alpha^2M_2^2}{n}\mE\left[\|\bv^k-\nabla f(\bx^k)\|\right]\\\nonumber
&+16\alpha^2M_2^2L\cdot n \mE\left[F(\ox^{k})-F(x^*)\right]\\\nonumber
&+4\alpha^2M_2^2\mE\left[\|\bg^{k}-\bW_{\infty}\bg^{k}\|^2\right].
\end{align}
\end{lemma}
\begin{proof}
By the update rule of $\bx^{k+1}$ in the \eqref{eq:alg}, we have
\begin{equation}\label{a39}
\begin{aligned}
\bx^{k+1}-\bx^k=&\bW\bx^k-\alpha \bd^k-\bx^k\\
=&(\bW-I_{nd})(\bx^k-\bW_{\infty}\bx^k)-\alpha \bd^k,
\end{aligned}
\end{equation}
where we use the fact that $(\bW-I_{nd})\bW_{\infty}\bx^k=0$ in the last equality. Then we have
\begin{align}\label{a40}\nonumber
\|\bx^{k+1}-\bx^k\|^2&\le 2\|(\bW-I_{nd})(\bx^k-\bW_{\infty}\bx^k)\|^2+2\alpha^2\|\bd^k\|^2\\
&\le 8\|\bx^k-\bW_{\infty}\bx^k\|^2+2\alpha^2M_2^2\|\bg^k\|^2,
\end{align}
where we use the facts of $\|\bW-I_{nd}\|_2\le 2$ and $\|\bd^k\|^2=\sum_{i=1}^{n}\|H_i^kg_i\|^2\le M_2^2\|\bg^k\|^2$ in the last inequality. For the term $\|\bg^k\|^2$, we have
\begin{align}\label{a41}
\mE\left[\|\bg^k\|^2\right]&\le \mE\left[2\|\bg^k-\bW_{\infty}\bg^k\|^2+2\|\bW_{\infty}\bg^k\|^2\right]\\
&=\mE\left[2\|\bg^k-\bW_{\infty}\bg^k\|^2+2n\|\og^k\|^2\right]. \notag
\end{align}
For the term $\|\og^k\|^2$, we have
\begin{align}\label{a42}
\mE\|\og^k\|^2=&\mE\|\ov^k\|^2 \\
=&\mE\left[\|\ov^k-\overline{\nabla f}(\bx^k)\|^2+\|\overline{\nabla f}(\bx^k)\|^2\right] \notag\\
=&\mE\left[\frac{1}{n^2}\|\bv^k-\nabla f(\bx^k)\|^2+\|\overline{\nabla f}(\bx^k)\|^2\right], \notag
\end{align}
in which to derive the last equality we use the fact that $\mE\left[\mE\left[\left\langle v_i^k-\nabla f_i(x_i^k),v_j^k-\nabla f_j(x_j^k) \right\rangle |\mathcal{F}^k\right]\right]=0, ~\forall i\neq j$.  For the term $\|\overline{\nabla f}(\bx^k)\|^2$, we have
\begin{align}\label{a43}
\|\overline{\nabla f}(\bx^k)\|^2 \le& 2\|\overline{\nabla f}(\bx^k)-\nabla F(\ox^k)\|^2+2\|\nabla F(\ox^k)\|^2\\\nonumber
\le & \frac{2L^2}{n}\|\bx^k-\bW_{\infty}\bx^k\|^2+4L\left(F(\ox^k)-F(x^*)\right),
\end{align}
where we use \eqref{b15} and \eqref{ineq:nabla-xk-bar} to derive the last inequality.

Taking total expectation on both sides of \eqref{a40}, substituting \eqref{a41}--\eqref{a43} into the result and using the fact $8+8\alpha^2M_2^2L^2<8.01$ with the parameters in Theorem \ref{thom-parameter}, we obtain \eqref{a38} and complete the proof.
\end{proof}

The following lemma bounds the the variance of gradient estimators  $\mE\left[\|\bv^{k}-\nabla f(\bx^k)\|^2\right]$ at time $k$.
\begin{lemma}\label{lem-var}
	Under Assumptions \ref{asm-lips}--\ref{asm-Hbound}, consider the iterates $\{\bv^k\}$ generated by \eqref{eq:alg}. For all $k\geq 1$, we have
	\begin{align}\label{a44}
	&\mE\left[\|\bv^{k}-\nabla f(\bx^k)\|^2\right]\\\nonumber
	\leq& {4BL^2}\cdot\Big(\mE\left[\|\bx^{k}-\bW_{\infty}\bx^k\|^2\right]+\frac{2n}{L}\mE\left[F(\ox^{k})-F(x^*)\right]\\\nonumber
	&+\mE\left[\|\btau^{k}-\bW_{\infty}\btau^k\|^2\right]+\frac{2n}{L}\mE\left[F(\otau^{k})-F(x^*)\right]\Big).
	\end{align}
\end{lemma}
\begin{proof}
Defining $v_{i,l}^k=\Big(\nabla f_{i,l}(x_i^{k})-\nabla f_{i,l}(\tau_i^{k})\Big)+\nabla f_{i}(\tau_i^{k})$, which is the special case with $b_i=1$,  we have
\begin{equation}
v_i^k=\frac{1}{b_i}\sum_{l\in S^{k}_i} v_{i,l}^k.
\end{equation}
Under the condition $\mathcal{F}^k$, since $v_i^k$ is a random sampling of $\{v_{i,l}^k\}_{l=1}^{m_i}$ with size $b_i$ and without replacement, we have
\begin{equation}\label{a46}
\text{Var}(v_i^k)=\frac{m_i-b_i}{(m_i-1)b_i} \text{Var}(v_{i,l}^k)\le B\cdot \text{Var}(v_{i,l}^k),
\end{equation}
where we use the definition of non-sampling rate $B=\max_{i\in\{1,\cdots,n\}}\left\{\frac{m_i-b_i}{(m_i-1)b_i}\right\}$ in \eqref{p18}. Then \eqref{a46} implies
\begin{equation}\label{a47}
\begin{aligned}
\hspace{-2em}\mE[\|\bv^k-\nabla f(\bx^k)\|^2|\mathcal{F}^k]=\sum_{i=1}^{n}\text{Var}(v_i^k)
\le B \sum_{i=1}^{n}\text{Var}(v_{i,l}^k).
\end{aligned}
\end{equation}

By the definition of $v_{i,l}^k$, we have
\begin{align}\label{a48}
&\text{Var}(v_{i,l}^k) \\\nonumber
\leq&\frac{1}{m_i}\sum_{l=1}^{m_i}\left\|\nabla f_{i,l}(x_i^k)-\nabla f_{i,l}(\tau^k_i)\right\|^2\\\nonumber
\leq&\frac{4}{m_i}\sum_{l=1}^{m_i}\left\|\nabla f_{i,l}(x_i^k)-\nabla f_{i,l}(\ox^k)\right\|^2+\left\|\nabla f_{i,l}(\ox^k)-\nabla f_{i,l}(x^*)\right\|^2\\\nonumber
&\qquad +\left\|\nabla f_{i,l}(x^*)-\nabla f_{i,l}(\otau^k)\right\|^2+\left\|\nabla f_{i,l}(\otau^k)-\nabla f_{i,l}(\tau^k_i)\right\|^2\\\nonumber
\leq&\frac{4}{m_i}\sum_{l=1}^{m_i}L^2\left\|x_i^k-\ox^k\right\|^2+\left\|\nabla f_{i,l}(\ox^k)-\nabla f_{i,l}(x^*)\right\|^2\\\nonumber
&\qquad +\left\|\nabla f_{i,l}(x^*)-\nabla f_{i,l}(\otau^k)\right\|^2+L^2\left\|\otau^k-\tau^k_i\right\|^2\\\nonumber
=& 4L^2\left(\left\|x_i^k-\ox^k\right\|^2+\left\|\otau^k-\tau^k_i\right\|^2\right)+\frac{4}{m_i}\sum_{l=1}^{m_i}\\\nonumber
&\Big(\left\|\nabla f_{i,l}(\ox^k)-\nabla f_{i,l}(x^*)\right\|^2+\left\|\nabla f_{i,l}(x^*)-\nabla f_{i,l}(\otau^k)\right\|^2\Big).
\end{align}
With Lemma \ref{lem:Lsmooth}, we have
\begin{equation}\label{p59}
	\begin{split}
		&\left\|\nabla f_{i,l}(\ox^k)-\nabla f_{i,l}(x^*)\right\|^2\\
		\leq& 2L\left(f_{i,l}(\ox^k)- f_{i,l}(x^*) - \nabla f_{i,l}(x^*)^T(\ox^k -x^*) \right)
	\end{split}	
\end{equation}
By  taking the mean w.r.t. $l$ from 1 to $m_i$ and taking the sum w.r.t. $i$ from 1 to $n$, and using the fact that $\frac{1}{n m_i}\sum_{i=1}^n\sum_{l=1}^{m_i}\nabla f_{i,l}(x^*) = \nabla F(x^*)= 0$, from \eqref{p59} we get
\begin{align}\label{ineq:nablafij}
&\sum_{i=1}^{n}\frac{1}{m_i}\sum_{l=1}^{m_i}\left\|\nabla f_{i,l}(\ox^k)-\nabla f_{i,l}(x^*)\right\|^2\\\nonumber
\leq& 2nL\left(F(\ox^k)- F(x^*)\right).
\end{align}
We use similar derivation on $\left\|\nabla f_{i,l}(x^*)-\nabla f_{i,l}(\otau^k)\right\|^2$. Thus, combining \eqref{a47}, \eqref{a48} and \eqref{ineq:nablafij}, we have
	\begin{align}
		&\mE[\|\bv^k-\nabla f(\bx^k)\|^2]\\\nonumber
		\le & \mE\Big[4L^2B\left(\|\bx^{k}-\bW_{\infty}\bx^k\|^2+\|\btau^{k}-\bW_{\infty}\btau^k\|^2\right)\\\nonumber
		& + 8nLB\left(F(\ox^k)- F(x^*)\right)+8nLB\left(F(\otau^k)- F(x^*)\right)\Big],
	\end{align}
which completes the proof.
\end{proof}

Prior to bounding the variance of gradient estimators  $\mE\left[\|\bv^{k+1}-\nabla f(\bx^{k+1})\|^2\right]$ at time $k+1$, we need the following corollary which bounds the consensus error $\mE[\|\bx^{k+1}-\bW_{\infty} \bx^{k+1}\|^2]$ at time $k+1$.

\begin{corollary} \label{corol-x}
Under Assumptions \ref{asm-lips}--\ref{asm-Hbound},  when the parameters satisfy the conditions in Theorem \ref{thom-parameter},
consider the iterates $\{\bv^k\}$ generated by \eqref{eq:alg}. For all $k\geq 1$, we have
\begin{align}\label{a52}
		&\mE\left[\|\bx^{k+1}-\bW_{\infty}\bx^{k+1}\|^2\right]\\\nonumber
		\hspace{-1em}\leq& \left(\!1\!-\!\frac{0.99(1\!-\!\sigma^2)}{2}\!\right)\mE\left[\|\bx^{k}-\bW_{\infty}\bx^{k}\|^2\right]\!+\!\frac{0.01(1\!-\!\sigma^2)\alpha\mu M_1}{L^2}\\\nonumber
		&\Bigg(2\mE\left[\|\bg^k\!-\!\bW_{\infty}\bg^k\|^2\right]\!+\!2.05\gamma^2Ln\mE\left[F(\ox^k)-F(x^*)\right]\\\nonumber
		&\!+\!\frac{4L^2B\gamma^2}{n}\!\left(\mE\left[\|\btau^{k}\!-\!\bW_{\infty}\btau^k\|^2\right]\!+\!\frac{2}{L} n\mE\left[F(\otau^{k})\!-\!F(x^*)\right]\!\right)\!\Bigg).
\end{align}	
\end{corollary}
\begin{proof}
By substituting Lemma \ref{lem-var} into Lemma \ref{lem-x}, we have
\begin{align}\label{a52}
	&\mE\left[\|\bx^{k+1}-\bW_{\infty}\bx^{k+1}\|^2\right]\\\nonumber
	\leq& \left(\frac{1+\sigma^2}{2}+\frac{2\gamma^2\alpha^2M_2^2L^2}{1-\sigma^2}\left(1+\frac{4B}{n}\right)\right)\mE\left[\|\bx^{k}-\bW_{\infty}\bx^{k}\|^2\right]\\\nonumber
	&+\frac{2\alpha^2M_2^2}{1-\sigma^2}\cdot\Bigg(2\mE\left[\|\bg^k-\bW_{\infty}\bg^k\|^2\right]\\\nonumber
	&+2L\gamma^2(1+\frac{4B}{n})\cdot n\mE\left[F(\ox^k)-F(x^*)\right]\\\nonumber
	&+\!\frac{4L^2B\gamma^2}{n}\!\left(\!\mE\left[\|\btau^{k}\!-\!\bW_{\infty}\btau^k\|^2\right]+\frac{2}{L} n\mE\left[F(\otau^{k})\!-\!F(x^*)\right]\!\right)\!\Bigg).
\end{align}
Then, by substituting the parameters in Theorem \ref{thom-parameter}, we have
\begin{align*}
\frac{1+\sigma^2}{2}+\frac{2\gamma^2\alpha^2M_2^2L^2}{1-\sigma^2}(1+\frac{4B}{n}) & \le 1-\frac{0.99(1-\sigma^2)}{2}, \\
\frac{2\alpha^2M_2^2}{1-\sigma^2} & \le \frac{0.01(1-\sigma^2)\alpha\mu M_1}{L^2}, \\
2L\gamma^2(1+\frac{4B}{n}) & \le 2.05L\gamma^2.
\end{align*}
This completes the proof.
\end{proof}

With Corollary \ref{corol-x},  we bound the variance of gradient estimators  $\mE\left[\|\bv^{k+1}-\nabla f(\bx^{k+1})\|^2\right]$ at time $k+1$.
\begin{corollary}\label{coro-var}
Under Assumptions \ref{asm-lips}--\ref{asm-Hbound},  when the parameters satisfy the conditions in Theorem \ref{thom-parameter},
consider the iterates $\{\bv^k\}$ generated by \eqref{eq:alg}. For all $k\geq 1$, we have
\begin{align}\label{a49}
	&\mE\left[\|\bv^{k+1}-\nabla f(\bx^{k+1})\|^2\right]\\\nonumber
	\leq& 4L^2B\Big(\mE\left[\|\bx^{k}-\bW_{\infty}\bx^k\|^2\right]\!+\!\frac{3\alpha M_2^2}{L M_1}\mE\left[\|\bg^{k}-\bW_{\infty}\bg^k\|^2\right]\\\nonumber
	&+1.01\cdot\frac{2n}{L} \mE\left[F(\ox^{k})-F(x^*)\right]+1.01\mE\left[\|\btau^{k}-\bW_{\infty}\btau^k\|^2\right]\\\nonumber
	&+1.01\cdot\frac{2n}{L}\mE\left[F(\otau^{k})-F(x^*)\right]\Big).
\end{align}
\end{corollary}
\begin{proof}
	If $\mod (k+1,T)=0$, then $\bv^{k+1}=\nabla f(\bx^{k+1})$ and the proof is trivial. We consider $\mod (k+1,T)\neq 0$ and thus $\btau^{k+1}=\btau^k$. We have
	\begin{align}
	&\mE\left[\|\bv^{k+1}-\nabla f(\bx^{k+1})\|^2\right]\\\nonumber	
	\leq& {4L^2}{B}\Big(\mE\left[\|\bx^{k+1}-\bW_{\infty}\bx^{k+1}\|^2\right]+\frac{2n}{L} \mE\left[F(\ox^{k+1})-F(x^*)\right]\\\nonumber
	&+\mE\left[\|\btau^{k}-\bW_{\infty}\btau^k\|^2\right]+\frac{2n}{L} \mE\left[F(\otau^{k})-F(x^*)\right]\Big)\\\nonumber	
	\leq& {4L^2}{B}\cdot\mE\left[\|\bx^{k+1}-\bW_{\infty}\bx^{k+1}\|^2\right]\\\nonumber
	&+{4L^2}{B}\cdot\frac{2}{L}\Bigg( (1-\tilde{\mu}\alpha) n\mE\left[F(\ox^{k})-F(x^*)\right]\\\nonumber
	& +\!\frac{1.01\alpha M_2^2}{M_1}\!\Big(2L^2\mE\left[\|\bx^{k}-\bW_{\infty}\bx^{k}\|^2\right]\!\\\nonumber
    & \hspace{5.5em} +\!\frac{\gamma^2}{4}\mE\left[\|\bg^{k}\!-\!\bW_{\infty}\bg^{k}\|^2\right]\!\Big)\\\nonumber
	& +\frac{1.02\alpha M_2^2}{2nM_1}\mE\left[\|\bv^k-\nabla  f(\bx^k)\|^2\right]\!\Bigg) \\\nonumber
	&+4L^2B\left(\mE\left[\|\btau^{k}-\bW_{\infty}\btau^k\|^2\right]+\frac{2}{L} n\mE\left[F(\otau^{k})-F(x^*)\right]\right)\\\nonumber	
	\leq& 4L^2B\left(1-\frac{0.99(1-\sigma^2)}{2}+\frac{4.04\alpha M_2^2L}{M_1}+\frac{4.08\alpha M_2^2LB}{nM_1}\right)\\\nonumber
	&\times\mE\left[\|\bx^{k}-\bW_{\infty}\bx^{k}\|^2 \right]\\\nonumber
	&+4L^2B\left(\frac{0.02(1-\sigma^2)\alpha \mu M_1}{L^2}+\frac{2.02\alpha M_2^2}{L M_1}\right)\\\nonumber
	&\times \mE \left[\|\bg^k-\bW_{\infty}\bg^k\|^2\right]\\\nonumber
	&+4L^2B\!\left(\!1-\tu\alpha\!+\!\frac{4.08\alpha M_2^2LB}{ nM_1}\!+\!{0.011(1\!-\!\sigma^2)\alpha \mu M_1}\!\right) \\\nonumber
	&\times\frac{2n}{L}\mE \left[F(\ox^k)-F(x^*)\right]\\\nonumber
	&+4L^2B\left(1+\frac{4.08\alpha M_2^2 LB}{ n M_1}+ \frac{0.01(1-\sigma^2)\alpha \mu M_1}{L^2} \cdot \frac{4L^2 B}{n} \right)\\\nonumber
	&\times\left(\mE \left[\|\btau^k-\bW_{\infty} \btau^k\|^2\right]+\frac{2n}{L}\mE \left[\|F(\otau^k)-F( x^*)\|^2\right]\right).
	\end{align}
	In the derivation, we use Lemma \ref{lem-var} and $\btau^{k+1}=\btau^k$ in the first inequality. We use Lemma \ref{lem-F} in the second inequality. We use Corollary \ref{corol-x} and  Lemma \ref{lem-var} and regroup the results in the third inequality. Then, with the parameters in Theorem \ref{thom-parameter}, we can check that
\begin{align*}
1-\frac{0.99(1-\sigma^2)}{2}+\frac{4.04\alpha M_2^2L}{ M_1}+\frac{4.08\alpha M_2^2L B}{ nM_1} & \le 1, \\
\frac{0.02(1-\sigma^2)\alpha \mu M_1}{L^2}+\frac{2.02\alpha M_2^2}{L M_1} & \le \frac{3\alpha M_2^2}{L M_1}, \\
1-\tu\alpha\!+\!\frac{4.08\alpha M_2^2LB}{ nM_1}\!+\!{0.011(1\!-\!\sigma^2)\alpha \mu M_1} & \le 1.01, \\
1+\frac{4.08\alpha M_2^2 LB}{ n M_1}+ \frac{0.01(1-\sigma^2)\alpha \mu M_1}{L^2} \cdot \frac{4L^2 B}{n} & \le 1.01.
\end{align*}
This completes the proof.
\end{proof}

\subsubsection{\bf Step IV}
With the definition of optimality measure $\bu^k$, we are going to reorganize the bounds of the previous three steps into a compact form and specify the linear rate.
\begin{proposition}\label{prop-linear}
	Under Assumptions \ref{asm-lips}--\ref{asm-Hbound}, when the parameters satisfy the conditions in Theorem \ref{thom-parameter},
	consider the iterates $\{\bu^k\}$ generated by \eqref{eq:alg}. For all $k\geq 1$, we have
	\begin{equation}\label{a50}
	\bu^{k+1}\leq J_\alpha\bu^k+H_\alpha\tilde{\bu}^k.
	\end{equation}
	where we define $$\tilde{\bu}^k=\begin{bmatrix}
	\mE[\|\btau^k-\bW_{\infty} \btau^k\|^2] \\
	\frac{2n}{L}\mE\left[F(\otau^{k})-F(x^*)\right]\\0
	\end{bmatrix}.$$
	Consequently, we have
	\begin{equation}\label{a51}
	\bu^{(t+1)T}\leq \left((J_\alpha)^T+\sum_{t=0}^{T-1}(J_\alpha)^tH_\alpha\right)\bu^{tT},
	\end{equation}
	where $J_\alpha$, $H_\alpha \in \mathbb{R}^{3\times 3}$ are used to aggregate the constants and their specific forms can be found in the proof.
\end{proposition}
\begin{proof}
By combining the previous lemmas, we can bound $\bu^{k+1}$ in terms of only $\bu^k$ and $\tilde{\bu}^k$. By regrouping the right-hand side of \eqref{a52} in Corollary \ref{corol-x}, we specify the coefficients in Corollary \ref{corol-x} as
\begin{align}
&\mE\left[\|\bx^{k+1}-\bW_{\infty}\bx^{k+1}\|^2\right] \label{w55} \\
\le&\underbrace{\left(1-\frac{0.99(1-\sigma^2)}{2}\right)}_{J_{\alpha}{_{11}}}\mE\left[\|\bx^{k}-\bW_{\infty}\bx^{k}\|^2\right] \notag\\
&+\underbrace{0.011(1-\sigma^2)\gamma^2\alpha\mu M_1}_{J_{\alpha}{_{12}}} \cdot \frac{2n}{L}\mE\left[F(\ox^k)-F(x^*)\right] \notag\\
&+\underbrace{0.02\alpha\mu M_1}_{J_{\alpha}{_{13}}}\cdot \frac{1-\sigma^2}{L^2}\mE\left[\|\bg^k-\bW_{\infty}\bg^k\|^2|\mathcal{F}^k\right] \notag\\
&+\underbrace{\frac{0.04(1-\sigma^2)\gamma^2\alpha\mu M_1B}{n}}_{H_{\alpha}{_{11}}} \mE\left[\|\btau^{k}-\bW_{\infty}\btau^k\|^2\right] \notag\\
&+\underbrace{\frac{0.04(1-\sigma^2)\gamma^2\alpha\mu M_1B}{n}}_{H_{\alpha}{_{12}}} \cdot\frac{2n}{L} \mE\left[F(\otau^{k})-F(x^*)\right]. \notag
\end{align}

After substituting Lemma \ref{lem-var} into Lemma \ref{lem-F}, we specify the coefficients in Lemma \ref{lem-F} as
\begin{align}\label{a53}
&\frac{2n}{L}\mE\left[F(\ox^{k+1})-F(x^*)\right]\\\nonumber
\leq&
\underbrace{\frac{2}{L}\left(\frac{\alpha M_2^2\eta}{2nM_1}\!\cdot\! {4L^2}{B}\!+\!\frac{2.02\alpha M_2^2}{M_1}  L^2 \right)}_{J_{\alpha}{_{21}}}\mE\left[\|\bx^{k}-\bW_{\infty}\bx^{k}\|^2 \right]\\\nonumber
&+\underbrace{\left(\frac{\alpha M_2^2\eta}{2nM_1} {8L}{B}+(1-\tilde{\mu}\alpha) \right)}_{J_{\alpha}{_{22}}}\cdot \frac{2n}{L}\mE\left[F(\ox^{k})\!-\!F(x^*)\right]\\\nonumber
&+\underbrace{\frac{0.51\alpha\gamma^2 M_2^2L}{M_1(1-\sigma^2)}}_{J_{\alpha}{_{23}}}\cdot \frac{1-\sigma^2}{L^2}\mE\left[\|\bg^{k}-\bW_{\infty}\bg^{k}\|^2 \right] \\\nonumber
&+\underbrace{\frac{\alpha M_2^2\eta}{L nM_1}{4L^2}{B}}_{H_{\alpha}{_{21}}} \mE\left[\|\btau^{k}-\bW_{\infty}\btau^{k}\|^2 \right]\\\nonumber
&+\underbrace{\frac{\alpha M_2^2\eta}{nL M_1}{4L^2}{B}}_{H_{\alpha}{_{22}}}\cdot\frac{2n}{L}\mE\left[F(\otau^{k})\!-\!F(\btau^*)\right].\\\nonumber
\end{align}
After substituting Lemma \ref{lem-xk}, Lemma \ref{lem-var} and Corollary \ref{coro-var} into Lemma \ref{lem-g}, we specify the coefficients in Lemma \ref{lem-g} as
\begin{align}\label{a54}
&\frac{1-\sigma^2}{L^2}\mE\left[\|\bg^{k+1}-\bW_{\infty}\bg^{k+1}\|^2\right]\\\nonumber
\leq&
\underbrace{\left(32.04+32B+64\alpha^2M_2^2L^2\frac{B}{n}\right)}_{J_{\alpha}{_{31}}}\mE\left[\|\bx^{k}-\bW_{\infty}\bx^{k}\|^2 \right]\\\nonumber
&+\underbrace{\left(32\alpha^2M_2^2L^2+32.16B+64\alpha^2M_2^2L^2\frac{B}{n}\right)}_{J_{\alpha}{_{32}}} \cdot \\\nonumber
& \quad \frac{2n}{L}\mE\left[F(\ox^{k})-F(x^*)\right]\\\nonumber
&+\underbrace{\left(\frac{1+\sigma^2}{2} +\frac{16L^2\alpha^2M_2^2}{1-\delta^2}+\frac{4}{1-\sigma^2}{4L^2}{B}\frac{3\alpha M_2^2}{L M_1}\right) }_{J_{\alpha}{_{33}}} \cdot \\\nonumber &\quad \frac{1-\sigma^2}{L^2}\mE\left[\|\bg^{k}-\bW_{\infty}\bg^{k}\|^2 \right] \\\nonumber
&+\underbrace{\left(64\alpha^2M_2^2L^2\frac{B}{n}+32.16B\right)}_{H_{\alpha}{_{31}}} \mE\left[\|\btau^{k}-\bW_{\infty}\btau^{k}\|^2 \right]\\\nonumber
&+\underbrace{\left(64\alpha^2M_2^2L^2\frac{B}{n}+32.16B\right)}_{H_{\alpha}{_{32}}}\frac{2n}{L}\mE\left[F(\otau^{k})-F(\btau^*)\right].
\end{align}
Thus, we have
\begin{align}\label{f61}
\bu^{k+1}\le&
\begin{bmatrix}
J_{\alpha}{_{11}}& J_{\alpha}{_{12}} & J_{\alpha}{_{13}}\\ J_{\alpha}{_{21}}& J_{\alpha}{_{22}} & J_{\alpha}{_{23}}\\
J_{\alpha}{_{31}}& J_{\alpha}{_{32}} & J_{\alpha}{_{33}}\\
\end{bmatrix}
\bu^{k}\\\nonumber
&+
\begin{bmatrix}
H_{\alpha}{_{11}} & H_{\alpha}{_{12}} & 0\\
H_{\alpha}{_{21}}& H_{\alpha}{_{22}} & 0\\
H_{\alpha}{_{31}}& H_{\alpha}{_{32}} & 0\\
\end{bmatrix}
\begin{bmatrix}
\mE[\btau^k-\bW_{\infty} \btau^k] \\
\frac{2n}{L}\mE\left[F(\otau^{k})-F(x^*)\right]\\0
\end{bmatrix},
\end{align}
By recursion, from \eqref{f61} we get
\begin{align}\label{f62}
\nonumber
 \bu^{(k+1)T}&\le (J_{\alpha})^T\bu^{kT}+(J_{\alpha})^{T-1}H_{\alpha}\tilde{\bu}^{kT}+\cdots+ H_{\alpha}\tilde{\bu}^{kT+T-1}\\\nonumber
 &\le (J_{\alpha})^T\bu^{kT}+\left((J_{\alpha})^{T-1}H_{\alpha}+\cdots+(J_{\alpha})^0 H_{\alpha}\right)\tilde{\bu}^{kT}\\
 &\le \left((J_\alpha)^T+\sum_{l=0}^{T-1}(J_\alpha)^lH_\alpha\right)\bu^{kT},
\end{align}
where we use the fact that $\btau^{kT}=\cdots=\btau^{k(T+1)-1}=\bx^k$ for SVRG in the second inequality, and $\mE[\|\btau^{kT}-\bW_{\infty} \btau^{kT}\|^2]=\mE[\|\bx^{kT}-\bW_{\infty} \bx^{kT}\|^2]$ so that $\tilde{\bu}^{kT}=\bu^{kT}$ in the third inequality.
\end{proof}
The following lemma replaces the matrix coefficients $J_{\alpha}$ and $H_{\alpha}$ with their upper bounds $J_{\alpha,\beta}$ and $H_{\alpha, \beta}$, where $J_{\alpha}\le J_{\alpha,\beta}$ and $H_{\alpha}\le H_{\alpha, \beta}$.
\begin{lemma}
To simplify the notation, we define
	\begin{equation*}
	\beta=16B  \quad\text{and}\quad \zeta=\frac{M_1^2\mu^2}{M_2^2L^2}.
	\end{equation*}
Under Assumptions \ref{asm-lips}--\ref{asm-Hbound}, if the parameters satisfy the conditions \eqref{p18} in Theorem \ref{thom-parameter},
there exist $J_{\alpha,\beta}$ and $H_{\alpha,\beta}$ which are defined by
\begin{equation*}
J_{\alpha,\beta}=\begin{bmatrix}
1-\frac{0.99(1-\sigma^2)}{2}& 0.011(1-\sigma^2)\zeta\ta\gamma^2 & 0.02\zeta\ta\\ 4.1\ta&  1-0.96\zeta\ta & \frac{0.51\ta\gamma^2}{1-\sigma^2}\\
33& c & 1-\frac{0.99(1-\sigma^2)}{2}\\
\end{bmatrix}
\end{equation*}
and
\begin{equation*}
H_{\alpha,\beta}=\begin{bmatrix}
0.01\ta\beta\gamma^2(1-\sigma^2)& 0.01\ta\beta\gamma^2(1-\sigma^2) & 0\\
0.03\ta\zeta(1-\sigma^2)^2  & 0.03\ta\zeta(1-\sigma^2)^2 &0 \\
2.03\beta& 2.03\beta & 0\\
\end{bmatrix},
\end{equation*}
where  $c\triangleq0.162(1-\sigma^2)\ta\zeta+2.01\beta$.
We have
$$J_{\alpha}\le J_{\alpha,\beta}, ~ \text{and}~ H_{\alpha}\le H_{\alpha,\beta}.$$
Thus, if
\begin{equation}\label{q60}
\bu^{(t+1)T}\leq \left((J_{\alpha,\beta})^T+\sum_{t=0}^{T-1}(J_{\alpha,\beta})^lH_{\alpha,\beta}\right)\bu^{tT}
\end{equation}
converges linearly, then we have that $\bu^{(t+1)T}\leq \Big((J_{\alpha})^T+\sum_{t=0}^{T-1}(J_\alpha)^lH_\alpha\Big)\bu^{tT}$ also converges linearly.
\end{lemma}

\begin{proof}
It is easy to verify that the parameters satisfying  \eqref{ineq:stepsize} in Theorem \ref{thom-parameter} imply
\begin{equation}\label{w60}
\ta\leq \frac{(1-\sigma^2)^2}{200}\leq\frac{1}{200} \quad\text{and}\quad \beta\leq 0.1.
\end{equation}
By the definitions in Lemma \ref{lem-J}, we have
\begin{equation}\label{w61}
\alpha\mu M_1=\zeta \ta \quad \text{and}\quad \alpha^2M_2^2L^2=\zeta\ta^2\leq\ta^2.
\end{equation}

To get $J_{\alpha,\beta}$ and $H_{\alpha, \beta}$, we compute the upper bounds of all entries of $J_{\alpha}$ and $H_{\alpha}$ using \eqref{w60} and $\eqref{w61}$ as follows.  With \eqref{w55}, we have
\begin{align*}
&J_\alpha{_{12}}=0.011(1-\sigma^2)\alpha \mu M_1\gamma^2 \le 0.011(1-\sigma^2)\zeta \tilde{\alpha}\gamma^2 ,\\
&J_\alpha{_{13}}=0.02\alpha\mu M_1=0.02\zeta \tilde{\alpha},\\
&H_\alpha{_{11}}=H_\alpha{_{12}}=0.04(1-\sigma^2)\gamma^2\alpha\mu M_1\frac{B}{n}\le 0.01\gamma^2\tilde{\alpha}\beta(1-\sigma^2).
\end{align*}
With \eqref{a53}, since $\eta\leq 1.02$, we have
\begin{align*}
J_\alpha{_{21}}=&\frac{2}{L}\left(\frac{\alpha M_2^2\eta}{2nM_1} {4L^2}{B}+\frac{2.02\alpha M_2^2}{M_1}  L^2 \right)\leq 4.1\ta.
\end{align*}
Noticing that $\tu\alpha=0.99\zeta\ta$, we have
\begin{align*}
J_\alpha{_{22}}=&\frac{\alpha M_2^2\eta}{2nM_1} {8L}{B}+(1-\tilde{\mu}\alpha)\\
=&\frac{4\ta}{n}\cdot \frac{\mu}{L}B(\gamma^2+4\alpha LM_1) + 1 - 0.99\zeta\ta\\
\le& 0.03\zeta\ta + 1-0.99\zeta\ta \le 1 - 0.96\zeta\ta,\\
 J_\alpha{_{23}}=&\frac{0.51\alpha M_2^2L\gamma^2}{ M_1(1-\sigma^2)}\leq\frac{0.51\ta\gamma^2}{1-\sigma^2}.
\end{align*}
Using the fact that $\eta=\gamma^2+4\alpha LM_1$, we have
\begin{align*}
 H_\alpha{_{21}}=H_\alpha{_{22}}\leq&\frac{\alpha M_2^2\eta}{L M_1 n}\cdot 4L^2B\le 0.03\ta\zeta(1-\sigma^2)^2.
\end{align*}
With \eqref{a54}, we have
\begin{align*}
J_\alpha{_{31}}&= 32.04+32B+64\alpha^2M_2^2L^2\frac{B}{n}\leq 33,\\
J_\alpha{_{32}}&=32\alpha^2M_2^2L^2+32.16B+64\alpha^2M_2^2L^2\frac{B}{n}\\
&\le 0.162(1-\sigma^2)\tilde{\alpha}\zeta +2.01\beta \triangleq c,\\
J_\alpha{_{33}}&=\frac{1+\sigma^2}{2} +\frac{16L^2\alpha^2M_2^2}{1-\delta^2}+\frac{4}{1-\sigma^2}{4L^2}{B}\frac{3\alpha M_2^2}{L M_1}\\
&\le \frac{1+\sigma^2}{2}+\frac{0.08(1-\sigma^2)}{200}+\frac{3(1-\sigma^2)}{2000}\\
&\le 1-\frac{0.99(1-\sigma^2)}{2}, \\
H_\alpha{_{31}}&=H_\alpha{_{32}}=64\alpha^2M_2^2L^2\frac{B}{n}+32.16B\leq 2.03\beta.
\end{align*}
By replacing the  entries of $J_{\alpha} $ and $H_{\alpha}$ with their upper bounds above, we get $J_{\alpha,\beta}$ and $H_{\alpha,\beta}$ such that
$$J_{\alpha}\le J_{\alpha,\beta} ~ \text{and} ~ H_{\alpha}\le H_{\alpha,\beta}.$$

If $	\bu^{(t+1)T}\leq \left((J_{\alpha,\beta})^T+\sum_{t=0}^{T-1}(J_{\alpha,\beta})^lH_{\alpha,\beta}\right)\bu^{tT}$ converges linearly,  $\bu^{(t+1)T}\leq \left((J_{\alpha})^T+\sum_{t=0}^{T-1}(J_\alpha)^lH_\alpha\right)\bu^{tT}$ also converges linearly since all the terms are non-negative. This completes the proof.
\end{proof}
In the following, we will show that $$	\bu^{(t+1)T}\leq \left((J_{\alpha,\beta})^T+\sum_{t=0}^{T-1}(J_{\alpha,\beta})^lH_{\alpha,\beta}\right)\bu^{tT}$$ converges linearly. Since $J_{\alpha,\beta}$ is non-negative, we have that $$\sum_{l=0}^{T-1} (J_{\alpha,\beta})^{l} \leq \sum_{l=0}^{\infty} (J_{\alpha,\beta})^{l}=\left(I_3-J_{\alpha,\beta}\right)^{-1}.$$
Therefore, to prove $$\bu^{(t+1)T}\leq \left((J_{\alpha,\beta})^T+\sum_{t=0}^{T-1}(J_{\alpha,\beta})^lH_{\alpha,\beta}\right)\bu^{tT}$$ converges linearly, it is sufficient to prove that
$$\mathbf{u}^{(t+1) T} \leq\left((J_{\alpha,\beta})^{T}+\left(I_3-J_{\alpha,\beta}\right)^{-1} H_{\alpha,\beta}\right) \mathbf{u}^{t T}
$$ converges linearly.
The rest of the convergence analysis is to derive the conditions on the parameters $\alpha$, $B$ and $T$, such that the following inequality holds
$$
\rho\left((J_{\alpha,\beta})^{T}+\left(I_3-J_{\alpha,\beta}\right)^{-1} H_{\alpha,\beta}\right)<1.
$$
To do so, we will bound the spectral radius of $J_{\alpha,\beta}$ and $\left(I_3-J_{\alpha,\beta}\right)^{-1} H_{\alpha,\beta}$, respectively.
The following lemma in \cite{horn2012matrix} is a useful tool to bound the spectral radius of a matrix.
\begin{lemma}\label{lem-radius}
	Let $A \in \mathbb{R}^{d \times d}$ be non-negative and $x \in \mathbb{R}^{d}$ be positive. If $A x \leq \beta x$ for $\beta>0,$ then $\rho(A) \leq\| A\|_{\infty}^{x} \leq \beta$.
\end{lemma}
The following lemma bounds the spectral radius of $J_{\alpha,\beta}$.
\begin{lemma}\label{lem-J}
Under Assumptions \ref{asm-lips}--\ref{asm-Hbound}, when the parameters satisfy the conditions in Theorem \ref{thom-parameter}, we have
\begin{equation}\label{g70}
J_{\alpha, \beta} \cdot\bz\le \left(1-\frac{\zeta\ta}{2}\right)\bz, ~ \bz=[1;z_2;z_3],
\end{equation}
where
\begin{equation}\label{g71}
z_2=\frac{10}{\zeta}+\frac{1.2\gamma^2z_3}{\zeta(1-\sigma^2)}, \quad z_3=\frac{200(\zeta+\beta)}{\zeta(1-\sigma^2)},
\end{equation}
and thus
\begin{equation}\label{i77}
\rho(J_{\alpha,\beta})\le \| J_{\alpha,\beta} \|_{\infty}^{\bf z} \le  1-\frac{\zeta\ta}{2}.
\end{equation}
\end{lemma}
\begin{proof}
To prove \eqref{g70}, by substituting the definitions of $J_{\alpha,\beta}$ and $H_{\alpha,\beta}$, we know it is sufficient to prove
\begin{align}\label{g72}
0.011(1-\sigma^2)\zeta\ta\gamma^2 z_2 + 0.02\zeta\ta z_3 & \le\!\frac{0.99(1-\sigma^2)}{2}\!-\!\frac{\zeta\ta}{2}, \notag\\
4.1\ta+\frac{0.51\ta\gamma^2}{1-\sigma^2} z_3 & \le 0.46\zeta\ta z_2, \notag\\
33+c z_2 & \le 0.49(1-\sigma^2)z_3,
\end{align}
where we use $0.49(1-\sigma^2)z_3\le (\frac{0.99(1-\sigma^2)-\zeta\tilde{\alpha}}{2})z_3$.

With $\zeta\leq 1$ and $\ta< \frac{(1-\sigma^2)^2}{200}$, the first inequality in \eqref{g72} holds since
$$
z_2=\frac{10}{\zeta}+\frac{240\gamma^2(\zeta+\beta)}{\zeta^2(1-\sigma^2)^2}\leq \frac{10}{\zeta}+\frac{240\cdot 1.1\zeta}{\zeta^2(1-\sigma^2)^2}\leq \frac{280}{\zeta(1-\sigma^2)^2},
$$
where we use $\gamma^2\beta\leq \frac{\zeta(1-\sigma^2)^2}{10}$ and
$$
z_3\leq\frac{200(1+1/10)}{\zeta(1-\sigma^2)}\leq \frac{220}{\zeta(1-\sigma^2)},
$$
so that
\begin{align*}
&-\frac{0.99(1-\sigma^2)}{2}+0.011(1-\sigma^2)\zeta\ta\gamma^2 z_2 + 0.02\zeta\ta z_3 +\frac{\zeta\ta}{2}\\
\le&-0.495(1-\sigma^2)+0.06(1-\sigma^2)+\frac{1-\sigma^2}{400}\le 0.
\end{align*}
The second inequality in \eqref{g72} holds since
$$ 0.46\zeta z_2=0.46\zeta(\frac{10}{\zeta}+\frac{1.2\gamma^2z_3}{\zeta(1-\sigma^2)})\ge 4.6+\frac{0.55\gamma^2z_3}{1-\sigma^2}.$$
With $c=2.01\beta+0.162(1-\sigma^2)\zeta\ta$, the third inequality in \eqref{g72} holds since
$$
c z_2<2.01\beta z_2 + 0.162(1-\sigma^2)\zeta\ta\frac{280}{\zeta(1-\sigma^2)^2}\leq 2.01\beta z_2 + 0.3,$$
and
$$2.01\beta z_2=\frac{20.1\beta}{\zeta}+2.01\frac{1.2\gamma^2\beta z_3}{\zeta(1-\sigma^2)}\leq \frac{20.1\beta}{\zeta}+0.25(1-\sigma^2)z_3.
$$
where we use $\gamma^2\beta\leq \frac{\zeta(1-\sigma^2)^2}{10}$ again, so that
\begin{align*}
	&33+cz_2\leq33.3+\frac{20.1\beta}{\zeta}+0.25(1-\sigma^2)z_3\\
	&\leq \frac{33.3(\zeta+\beta)}{\zeta}+0.25(1-\sigma^2)z_3\leq 0.49(1-\sigma^2)z_3.
	\end{align*}
Thus, \eqref{g70} holds. By Lemma \ref{lem-radius}, we get \eqref{i77} and complete the proof.
\end{proof}

The following lemma bounds the determinant of $I_3-J_{\alpha,\beta}$, which will be used to compute  $\left(I_3-J_{\alpha,\beta}\right)^{-1}$.
\begin{lemma}\label{lem-det}
Under Assumptions \ref{asm-lips}--\ref{asm-Hbound}, when the parameters satisfy the conditions in Theorem \ref{thom-parameter}, the determinant of matrix $I_3-J_{\alpha,\beta}$ satisfies
\begin{equation}
\det(I_3-J_{\alpha,\beta})\ge \frac{(1-\sigma^2)^2\zeta\ta}{6}.
\end{equation}
\end{lemma}
\begin{proof}
By the definition of $J_{\alpha,\beta}$, we compute the determinant of $I_3-J_{\alpha,\beta}$ as
\begin{equation}\label{g77}
\begin{aligned}
&\det(I_3-J_{\alpha,\beta})\\\nonumber
&= 0.495(1-\sigma^2)\left(0.96\zeta\ta \cdot 0.495(1-\sigma^2)-\frac{0.51\ta \gamma^2c }{1-\sigma^2}\right)\\
&+0.011(1-\sigma^2)\zeta\ta\gamma^2\left(-4.1\ta \times  0.495(1-\sigma^2)-\frac{16.83\ta\gamma^2 }{1-\sigma^2}\right)\\
&-0.02\zeta\ta\left(4.1\ta c+33\times 0.96\zeta \ta\right)\\
&\ge0.173(1-\sigma^2)^2\zeta\ta> \frac{(1-\sigma^2)^2\zeta\ta}{6},
\end{aligned}
\end{equation}
where we use \eqref{w60}, $c=2.01\beta+0.162(1-\sigma^2)\zeta\ta\leq 0.21$ and $\gamma^2c\leq 0.202\zeta(1-\sigma^2)^2$ in the first inequality. Thus, we complete the proof.
\end{proof}

\begin{theorem}\label{theo-linear}
Under Assumptions \ref{asm-lips}--\ref{asm-Hbound}, when the parameters satisfy the conditions in Theorem \ref{thom-parameter}, we have
\begin{equation}\label{h80}
\ta\le \frac{(1-\sigma^2)^2}{200}, ~ \beta\le \frac{1}{10}\min\left\{1, \frac{(1-\sigma^2)^2\zeta}{\gamma^2}\right\},
\end{equation}
and
\begin{equation}
	T\ge\frac{2\log(280/(\zeta(1-\sigma^2)^2))}{\zeta\ta}.
\end{equation}
Thus, for $\bq=[1; 10; \frac{200(\zeta+\beta)}{1-\sigma^2}]$, we have
\begin{equation}\label{h81}
\|\bu^{(t+1)T}\|_{\infty}^{\bf q}\leq  0.9  \|\bu^{tT}\|_{\infty}^{\bf q}.
\end{equation}
\end{theorem}
\begin{proof}
The adjoint matrix of $(I_3-J_{\alpha,\beta})$ satisfies
\begin{align}\label{g78}
&\operatorname{adj} (I_3-J_{\alpha,\beta})\\
\hspace{-1.5em}\leq &
\begin{bmatrix}
0.5\zeta \ta &0.011\zeta \ta & 0.03\zeta\ta^2\\
\frac{19\ta}{1-\sigma^2}&0.26(1-\sigma^2)^2 & 0.26\ta((1-\sigma^2)^2\zeta+\gamma^2)\\
33.1\ta(\zeta+\beta)& (1-\sigma^2)(\beta+\zeta)& 0.5(1-\sigma^2)\zeta\ta
\end{bmatrix},\nonumber
\end{align}
where we use \eqref{w60} in the inequality.
By the definitions of $H_{\alpha,\beta}$ and $\bq$, we compute
\begin{align}\label{w70}
H_{\alpha,\beta}\cdot\bq =& (1+q_2) \begin{bmatrix}
0.01\ta\beta\gamma^2 \\ 0.03\ta\zeta(1-\sigma^2)^2 \\2.03\beta
\end{bmatrix}\\
\leq& 11 \begin{bmatrix}
0.001\ta\zeta(1-\sigma^2)^2 \\ 0.03\ta\zeta(1-\sigma^2)^2 \\2.03\beta
\end{bmatrix}.\nonumber
\end{align}
With \eqref{g78} and \eqref{w70}, we have
\begin{align}
\hspace{-1em}\operatorname{adj} (I_3-J_{\alpha,\beta}) H_{\alpha,\beta} \cdot \bq
\leq 11\!\begin{bmatrix}
0.01\zeta\ta^2\\0.12\ta\zeta(1-\sigma^2)^2\\1.5\ta(1-\sigma^2)\zeta(\zeta+\beta)
\end{bmatrix},
\end{align}
where we use the fact that $((1-\sigma^2)^2\zeta+\gamma^2)\beta\leq0.2(1-\sigma^2)^2\zeta$. Thus, by substituting Lemma \ref{lem-det}, we have
\begin{align}\label{w72}
&(I_3-J_{\alpha,\beta})^{-1}H_{\alpha,\beta} \cdot\bq\\
\leq& \frac{11\ta\zeta(1-\sigma^2)^2}{\det(I_3-J_{\alpha,\beta})} \begin{bmatrix}
0.0001\\0.12\\\frac{1.5(\zeta+\beta)}{1-\sigma^2}
\end{bmatrix}\nonumber\\
\leq& 66\!
\begin{bmatrix}
0.0001\\0.12\\\frac{1.5(\zeta+\beta)}{1-\sigma^2}
\end{bmatrix}
\leq 0.8\begin{bmatrix}
	1\\10\\\frac{200(\zeta+\beta)}{1-\sigma^2}
\end{bmatrix}=0.8\bq. \nonumber
\end{align}

Lemma \ref{lem-radius} and \eqref{w72} imply that
\begin{equation}\label{i87}
\begin{split}
\rho\left((I_3-J_{\alpha,\beta})^{-1}H_{\alpha,\beta}\right)
\leq 0.8.
\end{split}
\end{equation}
Since $\frac{10}{\zeta}\leq z_2\leq \frac{280}{\zeta(1-\sigma^2)^2}$, we have
$$
\bq\leq\bz\leq \frac{28}{\zeta(1-\sigma^2)^2}\cdot\bq.
$$
According to \cite{horn2012matrix}, this yields
$$
\| (J_{\alpha})^T\|_{\infty}^{\bq}\leq \frac{28}{\zeta(1-\sigma^2)^2}\cdot\| (J_{\alpha})^T\|_{\infty}^{\bz}.
$$
Taking norm $\|\cdot\|^{\bq}_{\infty}$ on both sides of \eqref{q60}, and then  substituting \eqref{i77} and \eqref{i87}, we have
\begin{equation*}
\begin{aligned}
\|\bu^{(t+1)T}\|_{\infty}^{\bf q}&\leq  \| (J_\alpha)^T+\sum_{t=0}^{T-1}(J_\alpha)^lH_\alpha\|_{\infty}^{\bf q}\cdot\|\bu^{tT}\|_{\infty}^{\bf q} \\
&\leq\left(\| (J_{\alpha})^{T}\|_{\infty}^{\mathbf{q}}+0.8\right)\left\|\mathbf{u}^{t T}\right\|_{\infty}^{\mathbf{q}} \\
& \leq\left(\frac{28}{\zeta(1-\sigma^2)^2}\cdot\left(\| J_{\alpha}\|_{\infty}^{\bz}\right)^{T}+0.8\right)\left\|\mathbf{u}^{t T}\right\|_{\infty}^{\mathbf{q}} \\
& \leq\left( \frac{28}{\zeta(1-\sigma^2)^2}\cdot\exp \left\{-\frac{\zeta\ta T}{2} \right\}+0.8\right)\left\|\mathbf{u}^{t T}\right\|_{\infty}^{\mathbf{q}},
\end{aligned}
\end{equation*}
where we use the fact that $\|( J_{\alpha})^T\|_{\infty}^{\bz}\le\left(\| J_{\alpha}\|_{\infty}^{\bz}\right)^{T}$ and $1+a \leq \exp \{a\}, \forall a \in \mathbb{R}$. By setting $T\ge\frac{2\log(280/(\zeta(1-\sigma^2)^2))}{\zeta\ta}$ in the last inequality above, we get \eqref{h81} and complete the proof.
\end{proof}
By Theorem \ref{theo-linear}, we know that  Algorithm \ref{alg-framework} converges linearly to the optimum at the rate of 0.9, which is a constant.
\section{Conclusions}
The aim of this work is to develop viable stochastic quasi-Newton methods for decentralized learning.  In Part I, we develop a general algorithmic framework where each node adopts a local inexact quasi-Newton direction that approaches the global one asymptotically. To be specific, each node uses gradient tracking to estimate the average of variance-reduced local stochastic gradients, and then constructs a local Hessian inverse approximation by the DFP and BFGS methods without incurring extra sampling or communication. Under the assumption that the local Hessian inverse approximation is positive definite with bounded eigenvalues, we prove that the general framework converges  linearly  to the exact solution.


\bibliographystyle{IEEEtran}
\bibliography{IEEEabrv_PnADMM}




%








\end{document}